\documentclass[10pt,reqno]{amsart}
\usepackage{amsmath,amssymb,latexsym,soul,cite,mathrsfs}
\usepackage{color,enumitem,graphicx}
\usepackage[colorlinks=true,urlcolor=blue,
citecolor=red,linkcolor=blue,linktocpage,pdfpagelabels,
bookmarksnumbered,bookmarksopen]{hyperref}
\usepackage[english]{babel}
\usepackage{enumitem}

\usepackage[left=1.8cm,right=1.8cm,top=2.1cm,bottom=2.1cm]{geometry}

\usepackage[hyperpageref]{backref}



\date{}
\newtheorem{theorem}{Theorem}[section]
\newtheorem{rmk}{Remark}[section]

\newtheorem{cor}{Corollary}[section]
\newtheorem{prop}{Proposition}[section]

\newtheorem{lem}{Lemma}[section]

\begin{document}
	\title[Singular Choquard elliptic problems involving two nonlocal nonlinearities]{ Singular Choquard elliptic problems involving two nonlocal nonlinearities via the nonlinear Rayleigh quotient}
	\vspace{1cm}
	
	\author{Edcarlos D. Silva}
	\address{Edcarlos D. Silva \newline  Universidade Federal de Goias, IME, Goi\^ania-GO, Brazil}
	\email{\tt edcarlos@ufg.br}
	
	\author{Marlos R. da Rocha}
	\address{Marlos R. da Rocha \newline Universidade Federal de Goias, IME, Goi\^ania-GO, Brazil }
	\email{\tt marlosrodrigues14@gmail.com}
	
	\author{Jefferson S. Silva}
	\address{Jefferson S. Silva \newline Universidade Federal de Goias, IME, Goi\^ania-GO, Brazil }
	\email{\tt 	jeffersonsantos.mat@gmail.com}

	
	\thanks{The first author was partially supported by CNPq with grants 309026/2020-2}
	\thanks{Corresponding author: Edcarlos D. Silva, email: edcarlos@ufg.br}

	\begin{abstract}
		In the present work we shall consider the existence and multiplicity
		of solutions for nonlocal elliptic singular problems where the nonlinearity is driven by two convolutions terms. More specifically, we shall consider the following Choquard type problem:		
		\begin{equation*}
			\left\{\begin{array}{lll}
				-\Delta u+V(x)u=\lambda(I_{\alpha_1}*a|u|^q)a(x)|u|^{q-2}u+\mu(I_{\alpha_2}*|u|^p)|u|^{p-2}u \\
				u\in H^1(\mathbb{R}^N)
			\end{array}\right.
		\end{equation*}
		where $\alpha_2<\alpha_1$; $\alpha_1,\alpha_2\in(0,N)$ and $0<q<1$; $p\in\left(2_{\alpha_2},2^*_{\alpha_2} \right)$. Recall also that $2_{\alpha_j}=(N+\alpha_j)/N$ and $2^*_{\alpha_j}=(N+\alpha_j)/(N-2), j=1,2$. Furthermore, for each $q\in(0,1)$, by using the Hardy-Littlewood-Sobolev inequality we can find a sharp parameter $\lambda^*> 0$ such that our main problem has at least two solutions using the Nehari method. Here we also use the Rayleigh quotient for the following scenarios $\lambda \in (0,  \lambda^*)$ and $\lambda = \lambda^*$. Moreover, we consider some decay estimates ensuring a non-existence result for the Choquard type problems in the whole space.

	\end{abstract}
	
	\maketitle
	
	\section{Introduction}
	
	In the present work, we shall investigate the Choquard problem with two convolution terms defined in the entire space. Our focus will be on analyzing the following nonlocal elliptic problem:
	\begin{equation}\label{p1}\tag{$P_\lambda$}
		\left\{\begin{array}{lll}
			-\Delta u+V(x)u=\lambda(I_{\alpha_1}*a|u|^q)a(x)|u|^{q-2}u+\mu(I_{\alpha_2}*|u|^p)|u|^{p-2}u \\
			u\in H^1(\mathbb{R}^N).
		\end{array}\right.
	\end{equation}
	Recall also that $p>q$ and $\lambda > 0, \mu > 0, \alpha_2<\alpha_1$; $\alpha_1,\alpha_2\in(0,N),N \geq 3$; $p\in\left(2_{\alpha_2},2^*_{\alpha_2} \right)$; $q\in\left(0,1 \right)$. Throughout this work, we consider $2_{\alpha_j}=(N+\alpha_j)/N$ and $2^*_{\alpha_j}=(N+\alpha_j)/(N-2)$ for $j =1, 2$. The potential $V : \mathbb{R}^N \rightarrow \mathbb{R}$ is a continuous function. Moreover, we shall consider some assumptions for the potentials $V$ and $a$. It is worth noting that the Riesz potential can be expressed as follows:
	$$I_\alpha(x)=\frac{A_\alpha(N)}{|x|^{N-\alpha}}, x\in\mathbb{R}^N\quad \text{and}\quad A_\alpha(N)=\frac{\Gamma(\frac{N-\alpha}{2})}{\Gamma(\frac{\alpha}{2})\pi^{\frac{N}{2} }2^{\alpha}}$$
	where $\alpha \in (0, N)$ and $\Gamma$ represents the Gamma function as described in \cite{Moroz1}. It is important to mention that there exist several important works on singular elliptic problems, see for instance Fulks-Maybee \cite{fulks}. After this pioneer work many attention of researchers has been considered in the last years. This research demonstrated that if we consider a bounded region $\Omega \subset \mathbb{R}^3$ occupied by an electrical conductor, the function $u(x, t)$, representing the temperature at point $x \in \Omega$ and time $t$, satisfies the following equation:
	$$cu_t-k\Delta u=\frac{E^2(x,t)}{t^{\gamma}}.$$
	Here, $E(x, t)$ describes the local voltage drop, $t^{\gamma}$ with $\gamma > 0$ represents the electrical resistivity, $c$ and $k$ correspond to the specific heat and thermal conductivity of the conductor, respectively. It is important to say that Choquard problems involving a singular nonlinearity have been also
	extensively studied in recent years. For example, the authors in \cite{chowang} investigate the following problem:
	\begin{eqnarray}\left\{ \begin{array}{lll}
			\mathcal{L}(u)=\lambda\frac{f(x)}{u^{\beta}}+\left(\displaystyle\int_{\mathbb{R}^N}\frac{g(y)|u(y)|^q}{|x-y|^{\mu}}dy\right)g(x)u^{q-1}\quad&&\text{in}\ \mathbb{R}^N \nonumber \\
			u>0 &&\text{in}\ \mathbb{R}^N \nonumber
		\end{array}\right.	
	\end{eqnarray}
	where $N\geq2$, $1 < p < N/s$ with $s\in(0, 1)$, $1 < q<\frac{p}{2}\cdot\frac{2N-\mu}{N-ps}$, $\mu\in(0,N)$ and $0<\beta<1$. In that work, the author studied the multiplicity of solutions for fractional Kirchhoff equations with
	Choquard and singular nonlinearities. Since the energy functional associated to this kind of problem in general is not differentiable on $D^{s,p}(\mathbb{R}^N)$, the usual critical point theory is not available. Thus, the authors use
	the Nehari manifold approach to get the existence of two solutions for the problem. Another example of a singular problem can be seen in \cite{sun}, where the authors studied the following problem:
	\begin{eqnarray}
		\Delta u+\lambda u^{\beta}+p(x)u^{-\gamma}&=&0, \quad \text{in}\ \Omega \nonumber \\
		u&>&0, \quad \text{in}\ \Omega \nonumber \\
		u&=&0,\quad \text{on}\ \partial\Omega
	\end{eqnarray}
	where $\Omega\subset\mathbb{R}^N$ is a bounded domain, $p:\Omega\rightarrow\mathbb{R}$ is a given non-negative non-trivial function in $L^2(\Omega)$, $1<\beta<2^*-1$, $0<\gamma<1$ are two constants, $2^*=\frac{2N}{N-2}$ is the standard critical exponent, $N\geq 3$ and $\lambda>0$ is a real parameter. The goal in that work is to show how variational methods can be used to
	establish some existence and multiplicity results for singular problems. The main idea in that work was investigate a suitable minimization problem for the functional $I_{\lambda}$ given by
	$$I_{\lambda}(u)=\frac{1}{2}\int_{\Omega}|\nabla u|^2 dx-\frac{\lambda}{\beta+1}\int_{\Omega}|u|^{\beta+1} dx-\frac{1}{1-\gamma}\int_{\Omega}p(x)|u|^{1-\gamma} dx.$$
	In particular, the authors found the combined effects of singular and superlinear nonlinearities
	changing considerably the structure of the solution set. In recent years, many semilinear elliptic problems have been considered with nonlocal terms and general nonlinearities. Several works have been dedicated on this subject, see \cite{claudianor, claudianor G, CPAA,cingolani,Xl,moroz2,Moroz01,seok,zhang}. Most of them are focus on the Choquard problem where there exists a non-homogeneous general function $F$. The proofs presented in these works rely on a Pohozaev type identity  for the following problem: \begin{equation}\label{int1}
		\begin{array}{lll}
			-\Delta u+u=(I_{\alpha}*|u|^p)|u|^{p-2}u \quad \text{in } \mathbb{R}^N.
		\end{array}
	\end{equation}
	Furthermore, the authors have also proved the regularity, positivity, and radial symmetry of solutions, along with the existence of ground state solutions, for the Choquard equation. For more details on this topic we refer the interested readers to \cite{claudianor2,claudianor3,Sc,4moroz01, Moroz1,31moroz01,fractional}. It is also important to mention that the Choquard term involving a singular problem has studied with the critical term. For example, in \cite{india} the authors studied the following problem:
	\begin{equation}
		-\Delta u=\lambda u^{-q}+\left(\int_{\Omega}\frac{|u|^{2^*_{\mu}}}{|x-y|^{\mu}}dy\right)|u|^{2^*_{\mu}-2}u, \quad u>0\text{ in } \Omega, \quad u=0\text{ on }\partial\Omega
	\end{equation}
	where $\Omega$ is a bounded domain in $\mathbb{R}^N$ with smooth
	boundary $\partial\Omega$, $N>2$, $\lambda>0$, $0<q<1$, $0<\mu<N$ and $2^*_{\mu}=\frac{2N-\mu}{N-2}$. In that paper, was study the multiplicity
	results with convex-concave type critical growth and singular nonlinearity. The main difficulty in that work is treat the
	singular nonlinearity along with critical exponent in the sense of Hardy-Littlewood-Sobolev inequality which is nonlocal. Furthermore, the associated energy functional is not anymore differentiable
	due to presence of singular nonlinearity. Hence, the usual minimax theorems are not applicable in general. Recall also that the critical exponent term being nonlocal give us one more difficulty in order to study the compactness of Palais-Smale sequences. Notice also that elliptic partial differential equations with singular nonlinearity have been also extensive studied by many authors, see for instance \cite{orsina,canino,haitao,laser,Giaco} and references therein. In all of these works, the existence of solutions to the singular problem has been established through some approximation techniques. In those works the solution arised from a working space which varies depending on the power of the singular term.
	
	The Choquard-type equations involving two Choquard terms have recently also considered by many works taking into account diferent types of nonlinearities. For instance, in \cite{choty}, the author investigates the following nonlocal problem:
	\begin{equation}
		-\Delta u+(V+\lambda)u=(I_{\alpha}*|u|^p)|u|^{p-2}u+\mu(I_{\alpha}*|u|^q)|u|^{q-2}u \quad\text{in}\quad \mathbb{R}^N
	\end{equation}
	having a prescribed mass $\int_{\mathbb{R}^N}u^2=a>0$,  where $\lambda\in\mathbb{R}$ will arise as a Lagrange multiplier, $N\geq 3$, $I_{\alpha}$ is the Riesz potential,
	$\alpha\in(0,N)$, $p\in(\overline{\alpha},2^*_{\alpha}]$, $q\in(\overline{\alpha},2^*_{\alpha})$, $\overline{\alpha}=(N+\alpha+2)/N$ is the mass critical exponent, $2^*_{\alpha}=(N+\alpha)/(N-2)$  is the
	Hardy–Littlewood–Sobolev upper critical exponent and $\mu > 0$ is a constant. 
	Under appropriate conditions on the potential $V$, the Choquard-type equation mentioned above admits a positive normalized ground state solution by using comparison arguments. Furthermore, assuming that $p = 2^*_{\alpha}$, the authors require a larger value for $\mu$ and the Hardy–Littlewood–Sobolev subcritical approximation method is employed. Moreover, the author introduces a new result regarding the regularity of solutions and establishes a Pohozaev identity for a more general Choquard-type equation. Here we also refer the reader to \cite{choty2} where the authors proved the existence of ground states for the following critical Choquard type problem:
	\begin{equation}
		-\Delta u +2u=\lambda u+\alpha(I_\alpha*|u|^q)|u|^{q-2}u+(I_{\mu}*|u|^{2^*_{\mu}})|u|^{2^*_{\mu}} \quad\text{in}\ \mathbb{R}^N
	\end{equation}
	with prescribed mass
	$\int_{\mathbb{R}^N} u^2=c^2$, where $N\geq 3$, $\alpha > 0$, $0 <\mu< N$, $I_{\mu}$ is the
	Riesz potential, $\lambda\in\mathbb{R}$, and $(2N-\mu)/N<q<2^*_{\mu}=(2N-\mu)/(N-2)$. The author establishes that the critical Hartree equation has normalized ground states and mountain-pass type solutions when perturbed by a $L^2$-subcritical term with an exponent that satisfies $(2N-\mu)/N < q < (2N-\mu+2)/N$. Furthermore, the authors the existence of ground states of mountain-pass type when the problem is perturbed by a $L^2$-critical or $L^2$-supercritical term. For further references on Choquard-type equations with two convolution terms, we refer the reader to \cite{choty3, choty4, choty5}.

	In the present work the main goal is to find sharp conditions on the parameter $\lambda$ such that Problem \eqref{p1}
	admits at least two nontrivial solutions whenever $\lambda\in(0,\lambda^*]$. The first difficult arises from the presence of the function $a$ in one of the convolutions terms. Hence, by using the Hardy-Littlewood-Sobolev inequality, we prove that our energy functional is well-defined and continuous. Moreover, the energy functional for Problem \eqref{p1} is non-differentiable, which creates an obstacle in order to prove the existence of weak solutions using the Nehari method. Here we overcome this difficulty by using the Ekland Variational Principle. It is important to stress that here we borrow some ideas discussed in \cite{sun}. Furthermore, the lack of differentiability is another difficulty in order to show that minimizer on the Nehari manifold is also a critical point to the energy functional. The main tool in this case is to consider the existence of a curve $f$ belonging to the Nehari manifold $\mathcal{N}^-$ and another curve $h$ belonging to the Nehari manifold $\mathcal{N}^+$, see Section \ref{pre} ahead. The existence of these curves, along with the application of the Ekeland Variational Principle and the Fatou's Lemma, ensure the existence of at least two weak solutions for our main problem.
	
	The second main difficulty is to establish the existence of a positive number $\lambda^* > 0$ such that for each $\lambda = \lambda^*$, the problem \eqref{p1} possesses at least two solutions.  Namely, we find at least two solutions $u \in \mathcal{N}^+$ and $v \in \mathcal{N}^-$ with $\lambda = \lambda^*$. The difficulty in ensuring these solutions arises from the fact that the set $\mathcal{N}^0$ is non-empty. The set $\mathcal{N}^0$ contains the all inflections points for the fibering map associated to the our energy function which implies that the Lagrange Multiplier Theorem does not work anymore in $\mathcal{N}^0$ for $\lambda = \lambda^*$. To overcome this problem, we establish a non-existence result for solutions $u \in \mathcal{N}^0$. The third main challenge encountered in the present work is the presence of the singular term which is also a nonlocal nonlinearity involving the Choquard term. This term is also linked to the potential $a$. This is another difficult in order to ensure existence of weak solutions. The same difficulty appears also for the non-existence results in the Nehari set $\mathcal{N}^0$ for $\lambda=\lambda^*$. To overcome this difficulty, we consider a extra assumption on $a$ where some fine decay estimates are considered the non-existence result.

	\subsection{Assumptions and main theorems}\label{assu}
	In order to investigate the existence and nonexistence of nontrivial weak solutions to problem \eqref{p1}, we shall explore the values of the parameters $\mu$ and $\lambda$. The main feature here is to establish some sharp conditions in order to apply Nehari method. In order to do that we shall consider the following assumptions:
	\begin{itemize}
		\item [$(h_1)$] The potential $V:\mathbb{R}^N\rightarrow\mathbb{R}$ is continuous and there exists a constant $V_0>0$ such that $V(x)\geq V_0$ for all $x\in\mathbb{R}^N$;
		\item [$(h_2)$] For each $M>0$ it holds that $|\{x\in\mathbb{R}^N : V(x)\leq M \}|<+\infty$;
		\item [$(h_3)$]The potential $a$ is a positive function such that $a\in L^s(\mathbb{R}^N)$ where $s=2N/(N+\alpha_1-qN)$.
	\end{itemize}
	Now, we define
	$$X=\left\{u\in H^1(\mathbb{R}^N): \int_{\mathbb{R}^N}V(x)u^2 dx<+\infty \right\}.$$
	Notice that $X$ is a Hilbert space endowed with the norm and inner product as follows:
	$$\|u\|^2=\int_{\mathbb{R}^N}(|\nabla u|^2+V(x)u^2) dx,\ \langle u,v \rangle=\int_{\mathbb{R}^N}(\nabla u\nabla v+ V(x)uv) dx,\quad u,v\in X.$$
	It is worthwhile to emphasize that the energy functional $J:X\rightarrow\mathbb{R}$ associated to problem \eqref{p1} is given by
	$$J(u)=\frac{1}{2}\|u\|^2-\frac{\lambda}{2q}\int_{\mathbb{R}^N}(I_{\alpha_1}*a|u|^q)a(x)|u|^qdx-\frac{\mu}{2p}\int_{\mathbb{R}^N}(I_{\alpha_2}*|u|^p)|u|^pdx.$$
	It is not hard to see that the energy functional $J: X \to \mathbb{R}$ is continuous. Furthermore, taking into account the singular term, we mention that $J$ has Gateaux derivative only in some directions. Recall also that a function $u\in X$ is a critical point for the functional $J$ if, and only if, $u \in X$ is a weak solution to the Choquard Problem \eqref{p1}. Namely, a function $u\in X$ is said to be a weak solution for problem \eqref{p1} whenever
	$$\langle u,\varphi \rangle -\lambda\int_{\mathbb{R}^N}(I_{\alpha_1}*a|u|^q)a(x)|u|^{q-2}u\varphi dx-\mu\int_{\mathbb{R}^N}(I_{\alpha_2}*|u|^p)|u|^{p-2}u\varphi dx = 0$$
	holds for all $\varphi \in X$. Although the functional $J$ is only of class $C^0(X,\mathbb{R})$, we can verify that there exists the Gateaux derivative in the direction $u$ as follows:
	$$J^\prime(u)u=\|u\|^2-\lambda \int_{\mathbb{R}^N}(I_{\alpha_1}*a|u|^q)a(x)|u|^qdx-\mu \int_{\mathbb{R}^N}(I_{\alpha_2}*|u|^p)|u|^pdx.$$
	Similary, we also mention that 
	$$J^{\prime\prime}(u)(u,u)=2\|u\|^2-2q\lambda \int_{\mathbb{R}^N}(I_{\alpha_1}*a|u|^q)a(x)|u|^qdx-2p\mu \int_{\mathbb{R}^N}(I_{\alpha_2}*|u|^p)|u|^pdx.$$
	It should be noted that the Sobolev embedding $X\hookrightarrow L^r(\mathbb{R}^N)$ is a continuous for all $r\in[2,2^*]$. Furthermore, for each $r\in[2,2^*)$, the embedding $X\hookrightarrow L^r(\mathbb{R}^N)$ is compact. For further information on this subject we refer the reader to \cite{wang}.
	Under these conditions, by using some ideas introduced in \cite{yavdat1}, we should use the Nehari method together with the nonlinear Rayleigh quotient for our main problem. Hence, we consider the following Nehari set:
	\begin{equation}
		\mathcal{N}:=\left\{u\in X\setminus\{0\}:\|u\|^2-\mu\int_{\mathbb{R}^N}(I_{\alpha_2}*|u|^p)|u|^q dx=\lambda\int_{\mathbb{R}^N}(I_{\alpha_1}*a|u|^q)a(x)|u|^qdx \right\}.
	\end{equation}
	Under these conditions, we will split the Nehari manifold $\mathcal{N}$ into three disjoint subsets as follows:
	\begin{eqnarray}\label{n3}
		\begin{split}
			\mathcal{N}^+=\{u\in\mathcal{N}:J^{\prime\prime}(u)(u,u)>0\}; \\
			\mathcal{N}^-=\{u\in\mathcal{N}:J^{\prime\prime}(u)(u,u)<0\}; \\
			\mathcal{N}^0=\{u\in\mathcal{N}:J^{\prime\prime}(u)(u,u)=0\}.
		\end{split}
	\end{eqnarray}

	In order to apply the nonlinear Rayleigh quotient we shall use some definitions. To begin with, we need to introduce the following set
	\begin{eqnarray}\label{varepsilon}\mathcal{A}=\left\{u\in X\setminus\{0\};\ \frac{\|u\|^2}{2}-\frac{\mu}{2p}\int_{\mathbb{R}^N}(I_{\alpha_2}*|u|^p)|u|^p dx=\frac{\lambda}{2q}\int_{\mathbb{R}^N}(I_{\alpha_1}*a|u|^q)a(x)|u|^q dx \right\}.
	\end{eqnarray}
	The main objective in this work is to find weak solutions to our main problem using the following minimization problems:
	\begin{equation}\label{c1}
		c_{\mathcal{N}^-}:=\inf\{J(u): u\in\mathcal{N}^-\};
	\end{equation}
	\begin{equation}\label{c2}
		c_{\mathcal{N}^+}:=\inf\{J(u): u\in\mathcal{N}^+\}.
	\end{equation}
	As a first step, we need to show that $c_{\mathcal{N}^-}$ and $c_{\mathcal{N}^+}$ are attained. This kind of result establishes the existence of a weak solution to our main problem. A significant challenge in this study is identify some conditions for the parameters $\mu>0$ and $\lambda>0$ where the minimizers in the Nehari manifold are critical points for the energy functional $J$. It is worth noting that the sets described in \eqref{n3} and \eqref{varepsilon}  are minimization problems. Furthermore, using the Nehari method and the fibering function together with the nonlinear Rayleigh quotient we shall prove that all minimizers in the Nehari manifold give us critical points for the energy functinal $J$. This kind of problem have been widely studied in the recent years, see \cite{CPAA,yavdat2,yavdat0,yavdat1}. It is important to mention that

	$$u\in\mathcal{N} \quad\text{if and only if}\quad\lambda=\frac{\|u\|^2-\mu \displaystyle\int_{\mathbb{R}^N} (I_{\alpha_2}*|u|^p)|u|^p dx}{\displaystyle\int_{\mathbb{R}^N} (I_{\alpha_1}*a|u|^q)a(x)|u|^q dx}.$$
	Similarly, we deduce the following assertion
	$$J(u)=0\quad\text{if and only if}\quad\lambda=\frac{\frac{1}{2}\|u\|^2-\frac{\mu}{2p}\displaystyle\int_{\mathbb{R}^N} (I_{\alpha_2}*|u|^p)|u|^p dx}{\frac{1}{2q}\displaystyle\int_{\mathbb{R}^N} (I_{\alpha_1}*a|u|^q)a(x)|u|^q dx}.$$
	Under these conditions, we consider the functionals $R_n,R_e:X\setminus\{0\}\rightarrow\mathbb{R}$ as follows
	\begin{eqnarray}\label{Rn}R_n(u)=\frac{\|u\|^2-\mu \displaystyle\int_{\mathbb{R}^N} (I_{\alpha_2}*|u|^p)|u|^p dx}{\displaystyle\int_{\mathbb{R}^N} (I_{\alpha_1}*a|u|^q)a(x)|u|^q dx}
	\end{eqnarray}
	and
	\begin{eqnarray}\label{Re}	
		R_e(u)=\frac{\frac{1}{2}\|u\|^2-\frac{\mu}{2p}\displaystyle\int_{\mathbb{R}^N} (I_{\alpha_2}*|u|^p)|u|^p dx}{\frac{1}{2q}\displaystyle\int_{\mathbb{R}^N} (I_{\alpha_1}*a|u|^q)a(x)|u|^q dx}.
	\end{eqnarray}
	Based on our assumptions we can deduce that $R_n, R_e$ belong to $C^0(X,\mathbb{R})$. Here we also mention that due to the presence of a singular term the functionals given just above are not differentiable in general. On the other hand, we consider the following extremes:
	\begin{eqnarray}
		\lambda^*:=\inf_{u\in X\setminus\{0\}}\sup_{t>0} R_n(tu)\quad\text{and}\quad \lambda_*:=\inf_{u\in X\setminus\{0\}}\sup_{t>0} R_e(tu).
	\end{eqnarray}
	Hence, by using the minimization problem in the Nehari set $\mathcal{N}^+$, we can state our main first results as follows:
	\begin{theorem}\label{theorem1}
		Suppose $(h_1)-(h_3)$ holds. Assume also that $\lambda\in(0,\lambda^*)$. Then the problem \eqref{p1} admits at least one positive weak solution $v\in\mathcal{N}^+$ such that $c_{\mathcal{N}^+}=J(v) < 0.$
	\end{theorem}
	Now, by using the minimization problem in the Nehari set $\mathcal{N}^-$, we also prove the following result:
	\begin{theorem}\label{theorem2}
		Suppose $(h_1)-(h_3)$ holds. Assume also that $\lambda\in(0,\lambda^*)$. Then the problem \eqref{p1} admits at least one weak positive solution $u\in\mathcal{N}^-$ such that $c_{\mathcal{N}^-}=J(u)$. Furthermore, we obtain the following statements:
		\begin{itemize}
			\item[i)] For each $\lambda\in(0,\lambda_*)$ we obtain that $J(u)>0$;
			\item [ii)] For each $\lambda=\lambda_*$ we infer that $J(u)=0$;
			\item [iii)] For each $\lambda\in(\lambda_*,\lambda^*)$ we deduce that $J(u)<0$.
		\end{itemize}
	\end{theorem}

	At this stage, our objective is to guarantee the existence of weak solutions to problem \eqref{p1} with $\lambda=\lambda^*$. In order to do that we need to prove a nonexistence result for weak solutions $u \in \mathcal{N}^0$ for our main problem where $\lambda=\lambda^*$. In this case, we consider the following extra assumption:
	\begin{itemize}
		\item [$(h_4)$] Consider $\alpha_1\in(N-2,N)$; $p\in[2,\min\left\{2^*_{\alpha_2},2\alpha_1/(N-2)\right\})$,  $\alpha_2\in(N-4,\min\left\{N,(2p-q)(N-2)/2\right\})$ and $2(N-4)/(N-2)<2p-q$. The potential $a$ is a positive function such that $a\in L^{2/(2-q)}(\mathbb{R}^N)$ and $a\notin L^r(\mathbb{R}^N)$, where $r$ satisfies
		\begin{equation}\label{rr}
			\frac{2^* N}{2^*(\alpha_1-\alpha_2)+N(p-q)}<r<\min\left\{\frac{2^*}{p-q},\frac{2^*N}{pN+N(p-q)-2^*\alpha_2}\right\}.
		\end{equation}
	\end{itemize}
	Hence, we can written our next main result in the following form:
	
	\begin{theorem}\label{theorem3}
		Suppose $(h_1)-(h_4)$ holds. Assume also that $\lambda=\lambda^*$. Then, the problem \eqref{p1} does not admit any weak solution $u\in\mathcal{N}^0$. Furthermore, the problem \eqref{p1} has at least two weak solutions $u\in\mathcal{N}^-$ and $v\in\mathcal{N}^+$ such that
		$$c_{\mathcal{N}^-}=J(u)=\inf_{w\in\mathcal{N}^-} J(w); \quad c_{\mathcal{N}^+}=J(v)=\inf_{w\in\mathcal{N}^+}J(w).$$
	\end{theorem}
	
	\begin{rmk}
		It is worthwhile to mention that the cases $r=1$ and $r=2$ can be considered under specific constraints on $p$ and $q$ for the assumption $(h_4)$. The main idea here is to clarify how the hypothesis $(h_4)$ can be verified looking for some restriction on $p$ and $q$. More specifically, for $r=1$, we assume the following inequalities:
		\begin{eqnarray}
			\frac{2(N+\alpha_2-\alpha_1)}{N-2}<p-q; \quad \frac{2(N-4)}{N-2}<2p-q<\frac{2(N+\alpha_2)}{N-2}.
		\end{eqnarray}
		It is also essential that $1<2^*/(p-q)$. However, this estimate is already satisfied based on the assumptions about $p$ and $q$. Therefore, for $r=1$, we assume only the following estimates:
		\begin{equation}
			\frac{2(N+\alpha_2-\alpha_1)}{N-2}<p-q \quad \text{and} \quad \frac{2(N-4)}{N-2}<2p-q<\frac{2(N+\alpha_2)}{N-2}.
		\end{equation}
		On the other hand, assuming that $r=2$, we consider the following assumptions:
		\begin{equation}
			\frac{N-2(\alpha_1-\alpha_2)}{N-2}<p-q <\frac{N}{N-2}, \quad 2p-q<\frac{N+2\alpha_2}{N-2}. 
		\end{equation}
		Hence, assuming that $r=2$, we only assume the following conditions:	\begin{equation}
			\frac{N-2(\alpha_1-\alpha_2)}{N-2}<p-q<\frac{N}{N-2} \quad \text{and} \quad \frac{2(N-4)}{N-2}<2p-q<\frac{N+2\alpha_2}{N-2}.
		\end{equation}
	\end{rmk}
	\begin{rmk}
		For the case $r=2$ we can consider the hypothesis $(h_4)$ with some restrictions on the parameters $p$ and $q$. This is due to the fact that hypothesis $(h_4)$ implies that $a \in L^s(\mathbb{R}^N)$ where $s=2N/(N+\alpha_1-qN)$. Notice also that $q\in(0,1)$. It is not hard to see that there exists $\epsilon \in (0,1)$ such that
		\begin{eqnarray}
			\frac{2N}{N+\alpha_1-qN}&<&2 \quad\text{if}\quad q\in(0,\epsilon)\label{100}  \\
			\frac{2N}{N+\alpha_1-qN}&>&2 \quad\text{if}\quad q\in(1-\epsilon,1).
		\end{eqnarray}
		This situation can be a problem due to the interpolation law. As was told before, hypothesis $(h_4)$ implies that $a\not\in L^r(\mathbb{R}^N)$. Therefore, for the case $r=2$, we impose the condition $q\in(0,\alpha_1/N)$. This feature implies that inequality \eqref{100} is satisfied.
	\end{rmk}
	\begin{rmk}
		Let us consider a specific example for the parameters $p$ and $q$ where $r=1$ for hypothesis $(h_4)$. Here we also assume that $a$ does not belong to $L^r(\mathbb{R}^N)$. Consider also $N=3$ and $\alpha_1\in(1,3)$. Furthermore, by using the fact that $\alpha_2\in(0,\min\{3,(2p-q)/2\})$, we take $\epsilon>0$ small enough such that $\alpha_1=2+\epsilon$ and $\alpha_2=\epsilon$. Recall also that $p\in[2,\min\{2^*_{\alpha_2},2\alpha_1\})$. Under these conditions, we choose $p=3$ and $q=1/2$. Notice also that
		$2(3+\alpha_2-\alpha_1)<p-q$ and
		$2p-q< 2(3+\alpha_2)$ are also verified. The last assertion implies that hypothesis $(h_4)$ is satisfied. 
	\end{rmk}
	
	
	\subsection{Outline} 
	In Section \ref{pre}, we discuss preliminary results looking for the behavior of the fibering maps associated together with the Rayleigh quotient method. Section \ref{s3} is dedicated to the proof of our main results.
	
	\subsection{Notation}
	\begin{itemize}
		\item $2_{\alpha}= (N+\alpha)/N$ and $2^*_{\alpha}= (N+\alpha)/(N-2)$;	
		\item $A(u)=\displaystyle\int_{\mathbb{R}^N}(I_{\alpha_1}*a|u|^q)a(x)|u|^qdx$,\ \  $B(u)=\displaystyle\int_{\mathbb{R}^N}(I_{\alpha_2}*|u|^p)|u|^pdx$;
		\item In this work the norms $L^s(\mathbb{R}^N)$ and $L^{\infty}(\mathbb{R}^N)$ are denoted respectively by $\|\cdot\|_s$ and $\|\cdot\|_{\infty}$;
		\item 
		Define $\chi_{\Omega}$ as the characteristic function of the set $\Omega$;
		\item $A^\prime(u)\varphi=\displaystyle\int_{\mathbb{R}^N}(I_{\alpha_1}*a|u|^q)a(x)|u|^{q-2}u\varphi dx, B^\prime(u)\varphi=\displaystyle\int_{\mathbb{R}^N}(I_{\alpha_2}*|u|^p)|u|^{p-2}u\varphi dx, \varphi \in X;$
		\item 
		Define $(P_{\lambda^*})$ as the problem \eqref{p1} with $\lambda=\lambda^*$;
		\item $C,\overline{C_1},\tilde{C},C_1,C_2...$  denote positive constants (possibly different).
		\item $p^\prime= p/(p-1)$
	\end{itemize}

	\section{Preliminary results and variational setting}\label{pre}

	In the present section we shall consider the Nehari method and the nonlinear Rayleigh quotient to our main problem. To begin with, we shall present a well-known inequality.  Throughout this work, taking into account the convolutions terms, we consider the following result:
	\begin{prop}\label{hls}(Hardy-Littlewood-Sobolev inequality \cite{lieb})
		Let $t,r>1$ and $0<\alpha<N$ be such that $\frac{1}{t}+\frac{N-\alpha}{N}+\frac{1}{s}=2$. Let $\phi\in L^t(\mathbb{R}^N)$ and $\psi\in L^s(\mathbb{R}^N)$ be fixed functions. Then there exists a sharp constant $C=C(N,\alpha,t)>0$ such that 
		$$\int_{\mathbb{R}^N}\int_{\mathbb{R}^N}\frac{|\phi(x)\psi(y)|}{|x-y|^{N-\alpha}}dxdy\leq C\|\phi\|_t\|\psi\|_r.$$
	\end{prop}
	\begin{rmk}
		It follows from the last result that $I_{\alpha}*\phi\in L^{\frac{Ns}{N-\alpha s}}(\mathbb{R}^N)$ holds for all $\phi\in L^s(\mathbb{R}^N)$ and $s\in(1,\frac{N}{\alpha})$.
	\end{rmk}
	
	As observed in \cite{seok0} the Hardy-Littlewood-Sobolev inequality can also be formulated as follows:
	
	\begin{lem} \label{impotante} Consider $1\leq r<s<+\infty$ and $0<\alpha<N$ such that $$\frac{1}{r}-\frac{1}{s}=\frac{\alpha}{N}.$$
		Then for any $f\in L^r(\mathbb{R}^N)$, we obtain
		$$\left\|\frac{1}{|\cdot|^{N-\alpha}}*f\right\|_{s}\leq C(N,\alpha,r)\|f\|_{r}.$$
	\end{lem}
	As a result, we can demonstrate that our functional $J$ is well-defined. More specifically, we obtain the following result:
	\begin{prop}\label{hl}
		Suppose $(h_1)-(h_3)$ holds. Then the functions $A, B: X \to \mathbb{R}$ are well-defined.
	\end{prop}
	\begin{proof}
		Using the Hardy-Littlewood-Sobolev inequality, see Proposition \ref{hls}, we have that
		\begin{equation}
			\int_{\mathbb{R}^N}(I_{\alpha_1}*a|u|^q)a(x)|u|^q dx\leq C\|a|u|^q\|^2_t
		\end{equation}
		where $t=2N/(N+\alpha_1)$. Now, by using the H\"{o}lder's inequality, we obtain that
		\begin{equation}
			\|a|u|^q\|^2_t=\left(\int_{\mathbb{R}^N}[a(x)]^t|u|^{qt}\right)^{\frac{2}{t}}dx\leq\left[\left(\int_{\mathbb{R}^N}[a(x)]^{tr}dx\right)^{\frac{1}{r}}\left(\int_{\mathbb{R}^N}|u|^2dx\right)^{\frac{qt}{2}}\right]^\frac{2}{t}
		\end{equation}
		where $1/r+qt/2=1$. Hence, 
		\begin{equation}
			r=\frac{N+\alpha_1}{N+\alpha_1-qN}, \quad s=rt=\frac{2N}{N+\alpha_1-qN}.
		\end{equation}
		Notice that $a\in L^s(\mathbb{R}^N)$. Furthermore, by the continuous embedding $X\hookrightarrow L^r(\mathbb{R}^N)$ for $r\in[2,2^*]$, we conclude that 
		\begin{equation}\label{27}
			\int_{\mathbb{R}^N}(I_{\alpha_1}*a|u|^q)a(x)|u|^q dx	\leq C\|a|u|^q\|^2_t\leq[C_1\|a\|^t_s\|u\|^{qt}_2]^{\frac{2}{t}}\leq C_2\|a\|^2_s\|u\|^{2q}_2\leq C_3\|a\|^2_s\|u\|^{2q}<+\infty.
		\end{equation}
		Similarly, we have that
		\begin{equation}
			\int_{\mathbb{R}^N}(I_{\alpha_2}*|u|^p)|u|^p dx\leq C_4\||u|^p\|^2_t
		\end{equation}
		for $t=2N/(N+\alpha_2)$. Notice also that since $pt\in[2,2^*]$. Hence, we can use again the embedding $X\hookrightarrow L^r(\mathbb{R}^N)$ showing that
		\begin{equation}
			\||u|^p\|^2_t=\left(\int_{\mathbb{R}^N}|u|^{pt} dx\right)^\frac{2}{t}=\|u\|^{2p}_{pt}\leq C_5\|u\|^{2p}.
		\end{equation} 
		Under these conditions, we infer that
		\begin{equation}
			\int_{\mathbb{R}^N}(I_{\alpha_2}*|u|^p)|u|^p dx\leq C_5\|u\|^{2p}<+\infty.
		\end{equation}Hence, the energy functional $J$ is well-defined.
		This concludes the proof.
	\end{proof}
	
	\begin{rmk}\label{2o1} Based on the functionals $R_n$ and $R_e$ we observe that	
		$R_n(tu)=\lambda$ if and only if $ J^\prime(tu)tu=0$. Moreover, 
		$R_n(tu)>\lambda$ if and only if $ J^\prime(tu)tu>0 $. In the same way, we mention that
		$R_n(tu)<\lambda$ if and only if $ J^\prime(tu)tu<0$. Similarly, we obtain the following assertion $R_e(tu)=\lambda$ if and only if $ J(tu)=0$. Moreover, $R_e(tu)>\lambda$ if and only if $J(tu)>0$. Finally, we also mention that
		$R_e(tu)<\lambda$ if and only if $ J(tu)<0$. 
		
	\end{rmk}

	Now, we consider a relation between the functionals $R_n$ and $R_e$. More precisely,  we compare the energy functional $J$ with their derivatives. Hence, we can prove the following result:

	\begin{prop}
		Suppose that $(h_1)-(h_3)$ holds. Let $u\in X\setminus\{0\}$ be such that $R_e(tu)=\lambda$ for some $t>0$. Then, we obtain that $R^{\prime}_e(tu)u>0$ if and only if $J^{\prime}(tu)tu>0$. Moreover, we obtain that $R^{\prime}_e(tu)u<0$ if and only if $J^{\prime}(tu)tu<0$. Furthermore, we infer that $R^{\prime}_e(tu)u=0$ if and only if $J^{\prime}(tu)tu=0$.
		
	\end{prop}
	\begin{proof} 
		It is not hard to see that 
		$$R^\prime_e(tu)u=\frac{\frac{1}{t}J^\prime(tu)tu}{\frac{t^{2q}}{2q}A(u)}$$
		holds for each $u\in X\setminus\{0\}$ where $R_e(tu)=\lambda$ with $t>0$. Therefore, the desired result follows using the last identity.
	\end{proof}
	
	\begin{prop}\label{2.4}
		Suppose that $(h_1)-(h_3)$ holds. Let $u\in X\setminus\{0\}$ be such that $R_n(tu)=\lambda$ for some $t>0$. Then, we obtain that $R^{\prime}_n(tu)u>0$ if and only if $J^{\prime\prime}(tu)(tu,tu)>0$. Moreover, we infer that $R^{\prime}_n(tu)u<0$ if and only if $J^{\prime\prime}(tu)(tu,tu)<0$. Furthermore, we obtain that $R^{\prime}_n(tu)u=0$ if and only if $J^{\prime\prime}(tu)(tu,tu)=0$.
		
	\end{prop}
	\begin{proof} 
		Similarly, we can prove that
		$$R^\prime_n(tu)u=\frac{\frac{1}{t}J^{\prime\prime}(tu)(tu,tu)}{t^{2q}A(u)}$$
		holds for each $u\in X\setminus\{0\}$ such that $R_n(tu)=\lambda$ with $t>0$. Hence, by using the last identity, the proof for the desired result follows.
	\end{proof}
	
	Now, we shall investigate the fibering map $Q_n$, which is defined for each $t>0$ as follows
	$$Q_n(t)=R_n(tu)=\frac{t^{2-2q}\|u\|^2-t^{2p-2q}\mu B(u)}{A(u)}, t > 0.$$ 
	As a result, we obtain the following identity
	$$Q^\prime_n(t)=\frac{(2-2q)t^{1-2q}\|u\|^2-(2p-2q)t^{2p-2q-1}\mu B(u)}{A(u)}.$$
	It is easy to check that the unique critical point of $Q_n$ is given by
	$$t_n(u)=\left[\frac{(1-q)\|u\|^2}{(p-q)\mu B(u)}\right]^{\frac{1}{2p-2}}.$$
	Now, we observe that $\displaystyle\lim_{t\rightarrow 0}Q_n(t)=0$ and  $\displaystyle\lim_{t\rightarrow+\infty}Q_n(t)=-\infty$.
	It is important to stress that
	\begin{eqnarray}
		\lim_{t\rightarrow 0}\frac{Q_n(t)}{t^{2-2q}}&=&\frac{\|u\|^2}{A(u)}>0, \,\,
		\lim_{t\rightarrow 0}\frac{Q^\prime_n(t)}{t^{1-2q}}=\frac{2(1-q)\|u\|^2}{A(u)}>0, \nonumber \\
		\lim_{t\rightarrow+\infty}\frac{Q_n(t)}{t^{2p-2q}}&=&-\frac{\mu B(u)}{A(u)}<0, \,\,
		\lim_{t\rightarrow+\infty}\frac{Q^{\prime}_n(t)}{t^{2p-2q-1}}=\frac{-2(p-q)\mu B(u)}{A(u)}<0. \nonumber 
	\end{eqnarray}
	
	Consequently, the critical point $t_n(u)$ is unique which give us a global maximum point for the function $Q_n$. Under these conditions, we obtain the following identity:
	\begin{eqnarray}
		Q_n(t_n(u))=C_{p,q,\mu}\frac{\|u\|^{2\left(\frac{p-q}{p-1}\right)}B(u)^{\frac{q-1}{p-1}}}{A(u)} \nonumber
	\end{eqnarray}
	where 
	$$C_{p,q,\mu}=\left(\frac{1-q}{p-q}\right)^{\frac{1-q}{p-1}}\left(\frac{p-1}{p-q}\right)\mu^{\frac{q-1}{p-1}}.$$
	It follows that also that $Q^\prime_n(t)>0$ for each $t\in(0,t_n(u))$. Similarly, we observe that $Q^\prime_n(t)<0$ for each $t\in(t_n(u),+\infty)$.
	As a consequence, we can consider the functional $\Lambda_n:X\setminus\{0\}\rightarrow\mathbb{R}$ defined by
	$$\Lambda_n(u)=Q_n(t_n(u))=C_{p,q,\mu}\frac{\|u\|^{2\left(\frac{p-q}{p-1}\right)}}{B(u)^{\frac{1-q}{p-1}}A(u)}.$$
	Analogously, we define the function $Q_e(t)$ given by
	$$Q_e(t)=R_e(tu)=\frac{\frac{t^{2-2q}}{2}\|u\|^2-\frac{ t^{2p-2q}}{2p}\mu  B(u)}{\frac{1}{2q}A(u)}.$$ 
	Thus, we deduce that
	$$Q^\prime_e(t)=\frac{(1-q)t^{1-2q}\|u\|^2-\frac{(p-q)}{p}\mu t^{2p-2q-1}B(u)}{\frac{1}{2q}A(u)}.$$
	Notice also that $\displaystyle\lim_{t\rightarrow 0}Q_e(t)=0$ and $\displaystyle\lim_{t\rightarrow+\infty}Q_e(t)=-\infty$. Furthermore, we infer that 
	\begin{eqnarray}
		\lim_{t\rightarrow 0}\frac{Q_e(t)}{t^{2-2q}}&=&\frac{q\|u\|^2}{A(u)}>0,
		\lim_{t\rightarrow 0}\frac{Q^\prime_e(t)}{t^{1-2q}}=\frac{(1-q)2q\|u\|^2}{A(u)}>0, \,\, \nonumber \\
		\lim_{t\rightarrow+\infty}\frac{Q_e(t)}{t^{2p-2q}}&=&-\frac{\mu B(u)}{\frac{1}{2q}A(u)}<0, \,\,
		\lim_{t\rightarrow+\infty}\frac{Q^\prime_e(t)}{t^{2p-2q-1}}=-\frac{\mu\left(\frac{p-q}{p}\right)B(u)}{\frac{1}{2q}A(u)}<0. \nonumber
	\end{eqnarray}
	Therefore, the critical point $t_e(u)$ is unique proving that $t_e(u)$ is global maximum point for the function $Q_e$. It is easy to check that
	$$\left[\frac{(1-q)p\|u\|^2}{(p-q)\mu B(u)}\right]^{\frac{1}{2p-2}}=t_e(u)=p^{\frac{1}{2p-2}}t_n(u).$$
	Hence, we obtain that 
	\begin{eqnarray}
		Q_e(t_e(u))=\left[\frac{(1-q)p}{p-q}\right]^{\frac{1-q}{p-1}}q\left(\frac{p-1}{p-q}\right)\frac{1}{\mu^{\frac{1-q}{p-1}}}\frac{\|u\|^{2\left(\frac{p-q}{p-1}\right)}}{B(u)^{\frac{1-q}{p-1}}A(u)}. \nonumber
	\end{eqnarray}
	Consequently, we consider the functional $\Lambda_e:X\setminus\{0\}\rightarrow\mathbb{R}$ defined by: 
	\begin{eqnarray}
		\Lambda_e(u)=Q_e(t_e(u))=\left[\frac{(1-q)p}{p-q}\right]^{\frac{1-q}{p-1}}q\left(\frac{p-1}{p-q}\right)\frac{1}{\mu^{\frac{1-q}{p-1}}}\frac{\|u\|^{2\left(\frac{p-q}{p-1}\right)}}{B(u)^{\frac{1-q}{p-1}}A(u)}. \nonumber
	\end{eqnarray}
	It not hard to verify that $\Lambda_e(u)=p^{\frac{1-q}{p-1}}q\Lambda_n(u)$, see Figure 2.
	
	
	\begin{lem}\label{lema2.1}
		Suppose that $(h_1)-(h_3)$ holds. Assume that
		$$\Lambda_n(u):=Q_n(t_n(u))=C_{p,q,\mu}\frac{\|u\|^{2\left(\frac{p-q}{p-1}\right)}}{B(u)^{\frac{1-q}{p-1}}A(u)}$$
		Then we have the following statements:
		\begin{itemize}
			\item [i)] The functional $\Lambda_n$ is 0-homogeneous, that is, $\Lambda_n(tu) = \Lambda_n(u)$ for each $t > 0$; $u\in X\setminus\{0\}$.
			\item[ii)] The function
			$\Lambda_n$ is unbounded from above, that is, there exists a sequence $(w_k)_{k\in\mathbb{N}}\in X\setminus\{0\}$ such that $\Lambda_n(w_k)\rightarrow+\infty$ as $k\rightarrow+\infty$.
			\item[iii)] There exists $u\in X\setminus\{0\}$ such that $\lambda^*=\Lambda_n(u)=\displaystyle\inf_{v\in X\setminus\{0\}}\Lambda_n(v)$. Furthermore, we have that $\lambda^*>0$.
		\end{itemize}
	\end{lem}
	\begin{proof}  
		The proof of item $i)$ follows by using a standard argument. To establish the proof of item ii), we shall prove that $\Lambda_n(u)\geq C>0$. Let $(u_k)_{k\in\mathbb{N}}$ be a minimizing sequence. In view of Proposition \ref{hl} we infer that 
		\begin{eqnarray}\label{im}
			A(u_k)\leq C_1\|a\|^2_{s}\|u_k\|^{2q},\ \ 	B(u_k)\leq C_2\|u_k\|^{2p}. 
		\end{eqnarray}
		As a consequence, we deduce that
		\begin{eqnarray}
			\Lambda_n(u_k)&=&C_{p,q,\mu}\frac{\|u_k\|^{2\left(\frac{p-q}{p-1}\right)}}{B(u_k)^{\frac{1-q}{p-1}}A(u_k)}
			\geq\frac{C_{p,q,\mu}}{C_2 C_1\|a\|^2_{s}}\frac{\|u_k\|^{2\left(\frac{p-q}{p-1}\right)}}{\|u_k\|^{2p\left(\frac{1-q}{p-1}\right)}\|u_k\|^{2q}}=\frac{C_{p,q,\mu}}{C_2 C_1\|a\|^2_{s}}=K_1>0. \nonumber
		\end{eqnarray}
		Now, consider a sequence $(w_k)_{k\in\mathbb{N}}$ with $\|w_k\|=1$ such that $w_k\rightharpoonup 0$ and $(w_k)_{k\in\mathbb{N}}$ does not strong converges to zero. In light of the compact embedding already mentioned, we infer that $w_k\rightarrow 0$ in $L^r(\mathbb{R}^N)$ for each $r\in[2,2^*)$ and $w_k\rightarrow 0$ a.e. in $\mathbb{R}^N$. Hence, by using the Dominated Convergence Theorem, we prove that $A(w_k)\rightarrow0$ and $B(w_k)\rightarrow0$. Thus,
		$$\Lambda_n(w_k)=C_{p,q,\mu}\frac{\|w_k\|^{2\left(\frac{p-q}{p-1}\right)}}{B(w_k)^{\frac{1-q}{p-1}}A(w_k)}=\frac{C_{p,q,\mu}}{B(w_k)^{\frac{1-q}{p-1}}A(w_k)}\rightarrow +\infty\quad\text{as}\quad k\rightarrow+\infty.$$
		
		Now, we shall prove the item iii). Let $(u_k)_{k\in\mathbb{N}}$ be a minimizing sequence. Now, we consider the normalized sequence $w_k = u_k/\|u_k\|$. 
		Since $\Lambda_n$ is zero-homogeneous, we deduce that $(w_k)_{k\in\mathbb{N}}$ is also a minimizing sequence satisfying $\|w_k\|=1$. Thus, there exists $w \in X$ such that $w_k\rightharpoonup w$. As a consequence, by using the compact embedding $X\hookrightarrow L^r(\mathbb{R}^N)$, for each $r\in[2,2^*)$, we obtain that $w_k\rightarrow w$ in $ L^r(\mathbb{R}^N)$ where $r\in[2,2^*)$ and $w_k\rightarrow w$ a.e in $\mathbb{R}^N$. Recall also that
		$|w_k|\leq h_r$ where $h_r\in L^r(\mathbb{R}^N)$.
		Therefore, by using the Dominated Convergence Theorem, we obtain that $A(w_k)\rightarrow A(w)$ and $B(w_k)\rightarrow B(w)$
		as $k\rightarrow+\infty$. Recall also that $w_k\rightharpoonup w$ in $X$ where $w\neq0$. As a consequence, we obtain that  $\|w\|\leq\liminf_{k\rightarrow+\infty}\|w_k\|$. Under these conditions we obtain that 
		\begin{eqnarray}
			\lambda^*=\lim_{k\rightarrow+\infty}\Lambda_n(w_k)=\liminf_{k\rightarrow+\infty}C_{p,q,\mu}\frac{\|w_k\|^{2\left(\frac{p-q}{p-1}\right)}}{B(w_k)^{\frac{1-q}{p-1}}A(w_k)}\geq C_{p,q,\mu}\frac{\|w\|^{2\left(\frac{p-q}{p-1}\right)}}{B(w)^{\frac{1-q}{p-1}}A(w)}\geq\lambda^*. \nonumber
		\end{eqnarray}
		This ends the proof.
	\end{proof} 
	\begin{rmk}\label{lambda_e}
		Since $\Lambda_e(u)=p^{\frac{1-q}{p-1}}q\Lambda_n(u)$ we infer that the results such as lower bounds, attained minimum and continuity are also verfied to the functional $\Lambda_e$. Furthermore, we deduce that	$\Lambda_e(u)<\Lambda_n(u)$ holds for each $u \in X \{0\}$.
		In order to do that is enough to prove that $p^{\frac{1-q}{p-1}}q<1$. The last estimate is equivalent to the following inequality: 
		$$\left(\frac{1-q}{p-1}\right)\ln p+\ln q<\ln 1=0.$$
		Hence, we consider the function $f: \mathbb{R} \to \mathbb{R}$ given by $$f(q):=\left(\frac{1-q}{p-1}\right)\ln p+\ln q.$$
		It is easy to verify that $$f^\prime(q)=-\frac{\ln p}{p-1}+\frac{1}{q}.$$
		As a consequence,  $f^\prime(q) > 0$ holds for each $q \in (0, 1)$. Therefore, we obtain that $f(q) < f(1) = 0$.
	\end{rmk}
	
	\begin{rmk}\label{2.4}
		$R_n(tu)=R_e(tu)$ if, and only if $t=t_e(u)$. In fact, assuming that $R_n(tu)=R_e(tu)$, we obtain that
		\begin{eqnarray}
			\frac{1}{2q}[t^{2-2q}\|u\|^2-\mu t^{2p-2q}B(u)]&=&\frac{1}{2}t^{2-2q}\|u\|^2-\frac{\mu}{2p}t^{2p-2q}B(u). \nonumber
		\end{eqnarray} Thus, we obtain the following equality
		\begin{eqnarray}
			\left(\frac{1}{2q}-\frac{1}{2}\right)t^{2-2q}\|u\|^2+\left(\frac{1}{2p}-\frac{1}{2q}\right)\mu t^{2p-2q}B(u)=0. \nonumber 
		\end{eqnarray}
		Therefore, the last identity implies that
		\begin{eqnarray}
			t=\left[\left(\frac{1-q}{p-q}\right)p\frac{\|u\|^2}{\mu B(u)}\right]^{\frac{1}{2p-2}}=t_e(u). \nonumber
		\end{eqnarray}
	\end{rmk}
	\begin{prop}\label{2.3}
		Suppose $(h_1)-(h_3)$ holds. Then, for each $\lambda\in(0,\lambda^*)$ and $u\in X\setminus\{0\}$, the fibering map $\phi(t)=J(tu)$ has exactly two critical points denoted by $t^{n,+}(u)$ and $t^{n,-}(u)$ in such way that $0<t^{n,+}(u)<t_n(u)<t^{n,-}(u)$. Moreover, we consider the following statements:
		\begin{itemize}
			\item [i)] The function $t^{n,+}(u)$ is a local minimum point for the fibering map $\phi$ such that $t^{n,+}(u)u\in\mathcal{N}^+$. Furthermore, the functional $t^{n,-}(u)$ is a local maximum point for the fibering map $\phi$ such that $t^{n,-}(u)u\in\mathcal{N}^-$.
			\item[ii)] The functions $u\mapsto t^{n,+}(u)$ and $u\mapsto t^{n,-}(u)$ belongs to $C^0(X\setminus\{0\},\mathbb{R})$.
		\end{itemize}
	\end{prop}
	\begin{proof} i) Consider $0<\lambda<\lambda^*$ and $u\in X\setminus\{0\}$ a fixed function. Hence, we infer $R_n(t_n(u)u)=Q(t_n(u))\geq\lambda^*>\lambda$. Here was used the fact that $Q_n(t_n(u))=\max_{t>0}Q_n(t)$ and $\Lambda_n(u)=R_n(t_n(u)u)$. Therefore, the equation $Q_n(t)=R_n(tu)=\lambda$ has exactly two solutions. Here was used the fact that 
		$$\lim_{t\rightarrow 0}Q_n(t)=0\quad\text{and}\quad\lim_{t\rightarrow+\infty}Q_n(t)=-\infty.$$
		
		Now, we consider the two roots denoted as $t^{n,+}(u)$ and $t^{n,-}(u)$ which satisfy $0<t^{n,+}(u)<t_n(u)<t^{n,-}(u).$
		It is worth noting that $t^{n,+}(u)$ and $t^{n,-}(u)$ are critical points of the fibering map $\phi(t) = J(tu)$. Recall also that $R_n(t^{n,+}(u)u) = \lambda$ if, and only if $t^{n,+}(u)u \in \mathcal{N}$. Furthermore, we observe that 
		$R_n(t^{n,-}(u)u) = \lambda$ if, and only if $t^{n,-}(u)u \in \mathcal{N}$. Under these conditions, we obtain that $Q_n'(t^{n,+}(u))>0$ and $Q_n'(t^{n,-}(u))<0$. Now, by using Proposition \ref{2.4}, we infer that
		$$0<Q^\prime_n(t^{n,+}(u))=\frac{1}{t^{n,+}(u)}\frac{J^{\prime\prime}(t^{n,+}(u)u)(t^{n,+}(u)u,t^{n,+}(u)u)}{A(t^{n,+}(u)u)}.$$
		Hence, $t^{n,+}(u)u\in\mathcal{N}^+$. In the same way, we conclude that $t^{n,-}(u)u\in\mathcal{N}^-$ due to the fact that $Q_n'(t^{n,-}(u))<0$. This ends the proof of the item $i)$.
		
		Now we shall prove the item $ii)$. As was mentioned for each $\lambda\in(0,\lambda^*)$ we obtain that $\lambda<Q_n(t_n(u))=R_n(t_n(u)u)$ holds for all $u\in X\setminus\{0\}$. Hence, we can obtain two roots for the equation $R_n(tu)=\lambda$. These two roots satisfy $0<t^{n,+}(u)<t_n(u)<t^{n,-}(u)$ with $t^{n,+}(u)u\in\mathcal{N}^+$ and $t^{n,-}(u)u\in\mathcal{N}^-$. Consequently, we deduce that $\mathcal{N}=\mathcal{N}^+\cup\mathcal{N}^-$ for each $\lambda\in(0,\lambda^*)$. It is important to observe that $Q_n\in C^1(\mathbb{R}^+,X)$, $Q^\prime_n(t^{n,+})>0$ and $Q^\prime_n(t^{n,-})<0$. The desired result follows from the Implicit Function Theorem \cite{drabek}, that is, the maps $u\mapsto t^{n,+}(u)$ and $u\mapsto t^{n,-}(u)$ belong to $C^0(X\setminus\{0\},\mathbb{R})$ for any $\lambda\in(0,\lambda^*)$. In fact, we define the function $L^{\pm}:(0,+\infty)\times(X\setminus\{0\})\rightarrow\mathbb{R}$ given by $L^{\pm}(t,u)=J^\prime(tu)tu$. Hence, $L^{\pm}(t,u)=0$ if and only if $tu\in\mathcal{N}$. Furthermore, we infer that $\frac{\partial}{\partial t}L^{\pm}(t,u)\not=0$ for $(t,u)\in(0,+\infty)\times(X\setminus\{0\})$ such that $tu\in\mathcal{N}^{\pm}$. This ends the proof.
	\end{proof}
	Now, under certain conditions imposed on $\lambda>0$,  we shall verify that the set $\mathcal{N}^0$ is empty which implies that $\mathcal{N}=\mathcal{N}^+\cup\mathcal{N}^-$.
	\begin{prop}\label{2.55}
		Suppose $(h_1)-(h_3)$ holds. Then, for each $\lambda\in(0,\lambda^*)$, the set $\mathcal{N}^0$ is empty.
	\end{prop}
	\begin{proof} Suppose that there exists $u\in\mathcal{N}^0$ since $0<\lambda<\lambda^*$. In this case, we obtain that:
		\begin{equation}\label{aq}
			\lambda<\lambda^*=\inf_{w\in X\setminus\{0\}}\Lambda_n(w)\leq\Lambda_n(u)=\max_{t>0}Q_n(t)=R_n(t_n(u)u).
		\end{equation}
		According to Proposition \ref{2.4} we obtain that
		\begin{eqnarray}
			\frac{d}{dt}R_n(tu)=\frac{1}{t}\frac{J^{\prime\prime}(tu)(tu,tu)}{t^{2q}A(u)}. \nonumber
		\end{eqnarray}
		As a consequence, we establish the following statement:
		\begin{eqnarray}
			\frac{d}{dt}R_n(tu)|_{t=1}=0\quad\text{if, and only if}\quad t_n(u)=1. \nonumber 
		\end{eqnarray}
		Hence, we can conclude that $R_n(t_n(u)u)=\lambda$ which contradicts \eqref{aq}. Therefore, the set $\mathcal{N}^0$ is empty for each $\lambda\in(0,\lambda^*)$.
	\end{proof}
	\begin{rmk}
		
		Let $u\in X\setminus\{0\}$ be a fixed function. Assume also that $\Lambda_n(u) = \lambda$.  Hence, we can derive the following result:
		\begin{eqnarray}
			\lambda=\Lambda_n(u)=R_n(t_n(u)u)=Q_n(t_n(u)). 
		\end{eqnarray}
		This implies that $t_n(u)u\in\mathcal{N}$. Hence, 
		$$\frac{d}{dt}R_n(tu)|_{t=t_n(u)}=0\quad\text{if, and only if}\quad J^{\prime\prime}(t_n(u)u)(t_n(u)u,t_n(u)u)=0.$$
		Therefore, $t_n(u)u\in\mathcal{N}^0$ holds for each $\lambda \geq \lambda^*$. In particular, $\mathcal{N}^0$ is nonempty for $\lambda = \lambda^*$.
	\end{rmk}
	
	Now, using the same ideas discussed in the proof of Proposition \ref{2.3}, we can consider the function $Q_e$ proving the following result:
	\begin{prop}
		Suppose $(h_1)-(h_3)$ holds. Then, for each $\lambda\in(0,\lambda_*)$, there exist two points $t^{e,+}(u)$ and $t^{e,-}(u)$ such that $0<t^{e,+}(u)<t_e(u)<t^{e,-}(u)$ and $J(t^{e,-}(u)u)=J(t^{e,+}(u)u)=0$. Recall also that $Q^\prime_e(t^{e,-}(u))<0$ and $Q^\prime_e(t^{e,+}(u))>0$. Moreover, we obtain that the function $u\mapsto t^{e,+}(u)$ and $u\mapsto t^{e,-}(u)$ belongs to $C^0(X\setminus\{0\},\mathbb{R})$. Furthermore, we mention that
		$$0<t^{n,+}(u)<t^{e,+}(u)<t_n(u)<t_e(u)<t^{n,-}(u)<t^{e,-}(u)<+\infty$$
		holds for each $\lambda\in(0,\lambda_*)$.
	\end{prop}
	
	\begin{lem}\label{coercive}
		Suppose $(h_1)-(h_3)$ holds. Then the energy functional $J$ is coercive on the Nehari set $\mathcal{N}$ for every $\lambda>0$. In particular, the functional $J$ is bounded from below on $\mathcal{N}$.
	\end{lem}
	\begin{proof} Recall that $\|u\|^2-\lambda A(u)=\mu B(u)$ holds true for each $u\in\mathcal{N}$.
		Hence, we rewrite the functional $J$ as follows
		$$J(u)=\frac{1}{2}\|u\|^2-\frac{1}{2p}\left(\|u\|^2-\lambda A(u)\right)-\frac{\lambda}{2q}A(u)=\left(\frac{1}{2}-\frac{1}{2p}\right)\|u\|^2+\left(\frac{1}{2p}-\frac{1}{2q}\right)\lambda A(u).$$
		Now, by using \eqref{27} and Proposition \ref{hl}, we see that $A(u)\leq C_1\|a\|^2_{s}\|u\|^{2q}$. Here we observe also that $q \in (0,1)$ which imlies that $$J(u)\geq\left(\frac{1}{2}-\frac{1}{2p}\right)\|u\|^2+\left(\frac{1}{2p}-\frac{1}{2q}\right)\lambda C_1\|a\|^2_{s}\|u\|^{2q}.$$  Therefore
		$J(u)\rightarrow+\infty$ as $\|u\|\rightarrow+\infty$. This concludes the proof.
	\end{proof}
	\begin{lem}\label{3.2}
		Suppose $(h_1)-(h_3)$ holds. Also suppose that $\lambda\in(0,\lambda^*]$ holds. Then for each $u\in\mathcal{N}^-\cup\mathcal{N}^0$ and $\lambda\in(0,\lambda^*]$, there exists a constant $C=C(N,p,q)>0$ independent on $\lambda$ such that $\|u\|\geq C$. In particular, $\mathcal{N}^-$ and $\mathcal{N}^0$ are closed sets for $\lambda\in(0,\lambda^*]$.
	\end{lem}
	\begin{proof} Let $u\in\mathcal{N}^-$ be a fixed function with $\lambda\in(0,\lambda^*]$. In this case, we know that $t^{n,-}(u)=1$. Therefore,
		\begin{eqnarray}1=t^{n,-}(u)\geq t_n(u)=\left[\frac{(1-q)\|u\|^2}{(p-q)\mu B(u)}\right]^{\frac{1}{2p-2}}. \nonumber
		\end{eqnarray}
		Now, by using the Hardy-Littlewood-Sobolev inequality, see Proposition \ref{hls} and the continuous embedding  $X\hookrightarrow L^r(\mathbb{R}^N)$,$r\in[2,2^*]$, we infer that
		\begin{eqnarray}
			1\geq\left[\frac{(1-q)\|u\|^2}{(p-q)\mu B(u)}\right]^{\frac{1}{2p-2}}	\geq\left[\frac{(1-q)\|u\|^2}{(p-q)\mu C_2\|u\|^{2p}}\right]^{\frac{1}{2p-2}}=K_{p,q,\mu}\frac{1}{\|u\|} \nonumber
		\end{eqnarray} where $K_{p,q,\mu}=\left[(1-q)/[(p-q)\mu C_2]\right]^{\frac{1}{2p-2}}$. The last estimate implies that $\|u\|\geq K_{p,q,\mu}>0.$

		Now, let us assume that $u\in\mathcal{N}^0$. Hence, $t_n(u)=1$ and $R^\prime_n(u)u=0$. Thus, $u$ satisfies $\lambda^*=\Lambda_n(u)=R_n(t_n(u)u)=R_n(u)$. Hence, using the same ideas given just above, we obtain that
		$$1=t_n(u)=\left[\frac{(1-q)\|u\|^2}{(p-q)\mu B(u)}\right]^{\frac{1}{2p-2}}\geq K_{p,q,\mu}\frac{1}{\|u\|}.$$
		Once again, $\|u\|\geq C$ holds true for any $u\in\mathcal{N}^-\cup\mathcal{N}^0$. Now, taking into account the strong convergence and the last statement, we infer that $\mathcal{N}^-$ and $\mathcal{N}^0$ are closed sets for each $\lambda\in(0,\lambda^*]$. This concludes the proof.
	\end{proof}
	
	At this stage, we shall borrow some ideas discussed in \cite{sun}. Here we shall prove the existence of continuous curves in the Nehari sets $\mathcal{N}^+$ and $\mathcal{N}^-$. Namely, we shall prove the following important result:

	\begin{lem}\label{curva}
		Suppose $(h_1)-(h_3)$ holds. Let $u\in\mathcal{N}^-$ be a fixed function.  Then there exists $\epsilon>0$ and a continuous function $f: B_{\epsilon}(0) \subset X \to \mathbb{R}$ satisfying
		$$f(0)=1\quad\text{and}\quad f(w)(u+w)\in\mathcal{N}^-\quad\text{for all }w\in X \ \text{such that }\|w\|<\epsilon.$$
	\end{lem}
	\begin{proof}
		Let us define the function $F:\mathbb{R}\times X\rightarrow\mathbb{R}$ as follows:
		$$F(t,w)=t^{2-2q}\|u+w\|^2-\lambda A(u+w)-\mu t^{2p-2q}B(u+w).$$
		Let $u\in\mathcal{N}^-$ be a fixed function. As a consequence, we obtain
		\begin{equation}\label{N-}
			\|u\|^2-\lambda A(u)-\mu B(u)=0.
		\end{equation} 
		Notice that $F(1,0)=0$. Furthermore, we mention that 
		\begin{equation*}
			\frac{d}{dt} F(t,w)=(2-2q)t^{1-2q}\|u+w\|^2-(2p-2q)\mu t^{2p-2q-1}B(u+w).
		\end{equation*} 
		Thus, we obtain that 
		\begin{equation*}
			\frac{d}{dt}F(1,0)=(2-2q)\|u\|^2-(2p-2q)B(u).
		\end{equation*}
		Notice also that $J^{\prime\prime}(u)(u,u)=2\|u\|^2-\lambda 2q A(u)-\mu 2p B(u)<0$. Therefore, by using the estimates $\eqref{N-}$, we infer that
		$$J^{\prime\prime}(u)(u,u)=(2-2q)\|u\|^2-(2p-2q)\mu B(u)<0.$$
		It follows that $\frac{d}{dt}F(1,0)<0$. Hence, by using the Implicit Function Theorem \cite{drabek}, there exists $\overline{\epsilon}>0$ and a unique function $f:B_{\overline{\epsilon}}(0)\subset X\rightarrow V_{\overline{\epsilon}}\subset\mathbb{R}$ such that
		$$f(w)>0 \text{ and } f(0)=1\quad F(f(w),w)=0 \quad \text{for all }w\in B_{\overline{\epsilon}}(0).$$
		In particular, $f(w)(u+w)\in\mathcal{N}$. Furthermore, assuming that $\epsilon<\overline{\epsilon}$, we also obtain that
		$$\frac{d}{dt}F(f(w),w)<0\quad\text{for }\|w\|<\epsilon.$$
		It is not hard to see that 
		$$\frac{d}{dt}F(f(w),w)<0\quad\text{if, and only if}\quad J^{\prime\prime}(f(w)(u+w))(f(w)(u+w),f(w)(u+w))<0.$$
		Thus, $f(w)(u+w)\in\mathcal{N}^-$. This concludes the proof.
	\end{proof}
	Now, we can prove the existence of a curve in $\mathcal{N}^+$ in a similar way. Namely, we arrive at the following result:
	\begin{lem}\label{curva2}
		Suppose $(h_1)-(h_3)$ holds. Let $v\in\mathcal{N}^+$ be a fixed function. Then there exists $\epsilon>0$ and a continuous function $h:B_{\epsilon}(0) \subset X \to \mathbb{R}$ satisfying
		$$h(0)=1\quad\text{and}\quad h(w)(v+w)\in\mathcal{N}^+\quad\text{for all}\ w\in X \ \text{such that }\|w\|<\epsilon.$$
	\end{lem}

	In the sequel we shall prove that any minimizing sequence in $\mathcal{N}^-$ strong converges in $X$. More specifically, we can prove the following result:
	
	\begin{lem}\label{2.5}
		Suppose $(h_1)-(h_3)$ holds. Also assume that $\lambda\in (0,\lambda^*)$. Define $(v_k)_{k\in\mathbb{N}}\subset \mathcal{N}^-$ as a minimizing sequence. Then there exists $v\in X$ such that, up to a subsequence, $v_k\rightarrow v$ in $X$ where $v\in\mathcal{N}^-$. Moreover, we obtain that $c_{\mathcal{N}^-}=J(v)$.
	\end{lem}
	\begin{proof} Initially, we consider a minimizing sequence $(v_k)_{k\in\mathbb{N}}\subset\mathcal{N^-}$. It is important observe that $(v_k)_{k\in\mathbb{N}}$ is bounded. Here was used the fact that $J$ is coercive over the Nehari set $\mathcal{N}$. Therefore, up to a subsequence, we obtain that $v_k\rightharpoonup v$ for some $v \in X$. Hence,  by using the compact embedding $X\hookrightarrow L^r(\mathbb{R}^N)$, $r\in[2,2^*)$, we infer that $v_k\rightarrow v$ in $ L^r(\mathbb{R}^N)$, $v_k\rightarrow v$ a.e in $\mathbb{R}^N$. Recall also that $|v_k|\leq h_r$ where $h_r\in L^r(\mathbb{R}^N)$. Under these conditions, by using the Dominated Convergence Theorem, we know that $A(v_k)\rightarrow A(v)$ and $B(v_k)\rightarrow B(v)$. Now we claim that $v\neq 0$. The proof of this assertion follows arguing by contradiction. Let us assume that $v \equiv 0$. Hence, by using Lemma \ref{3.2} and the compact embedding listed just above, we obtain that
		$$0<C\leq\|v_k\|^2=\lambda A(v_k)+\mu B(v_k)\rightarrow 0$$
		as $k\rightarrow+\infty$. This is a contradiction proving that $v\not=0$. Now, by using Proposition \ref{2.3}, the fibering map $\phi(t)=J(tv)$ with $t\geq0$ admits a unique critical point $t^{n,-}(v)>0$ such that $t^{n,-}(v)v\in\mathcal{N}^-$. It remains to prove that $v_k \to v$ in $X$. The proof follows by contradiction. Let us assume that $v_k\not\rightarrow v$. In particular, we obtain $\|v\|<\liminf_{k\rightarrow+\infty}\|v_k\|.$
		Notice also that $v_k\in\mathcal{N}^-$ implies that
		$$J(v_k)\geq J(s v_k)\quad \text{for all}\ s\geq t^{n,+}(v_k).$$
		Now, we claim that $t^{n,-}(v)>t^{n,+}(v)$. The proof for this claim comes from the fact that $v\mapsto J^\prime(v)v$ is weakly lower semicontinuous. Hence, we obtain that 
		$$0=J^\prime(t^{n,-}(v)v)v<\liminf_{k\rightarrow+\infty}J^\prime(t^{n,-}(v)v_k)v_k,$$
		As a consequence, we know that $J^\prime(t^{n,-}(v)v_k)v_k>0$ for each $k$ large enough. It follows that $t^{n,-}(v)\in(t^{n,+}(v_k),t^{n,-}(v_k))$. Here was used the fact that $t \mapsto J(tv_k)$ is increasing on the interval $(t^{n,+}(v_k),t^{n,-}(v_k))$, see Figure \ref{figure1}. Now, by using the last inequality, we can deduce that
		$$c_{\mathcal{N}^-}\leq J(t^{n,-}(v)v)<\liminf_{k\rightarrow +\infty} J(t^{n,-}(v_k)v_k)\leq \liminf_{k\rightarrow+\infty} J(v_k)=c_{\mathcal{N}^-}.$$
		This is a contradiction proving that $v_k\rightarrow v$ in $X$. Under these conditions, by using  the strong convergence in $X$, we deduce that $\|v\|\geq C>0$, $J^\prime(v)v=0$, and $J^{\prime\prime}(v)(v,v) < 0$. Here was used also the fact that $\mathcal{N}^0$ is empty. Hence, we deduce that 
		$c_{\mathcal{N}^-}=\lim_{k\rightarrow+\infty}J(v_k)=J(v).$
		This ends the proof.
	\end{proof}
	\begin{prop}\label{weak1}
		Suppose $(h_1)-(h_3)$ holds. Assume also that $\lambda\in(0,\lambda^*)$. Then there exists a function $w\in\mathcal{N}^-$ obtained in the same way of Lemma \ref{2.5} such that $c_{\mathcal{N}^-}=J(w)$ and $w$ is a weak nontrivial solution for the problem \eqref{p1}.
	\end{prop}
	\begin{proof}
		According to  Lemmas \ref{coercive} and \ref{2.5} we obtain a function $v\in\mathcal{N}^-$ such that $c_{\mathcal{N}^-}=J(v)$. Recall that $\mathcal{N}^-$ is closed for $\lambda\in(0,\lambda^*)$. Hence, we can apply the Ekeland's Variational Principle \cite{drabek}, we obtain the existence of a minimizing sequence $(w_k)_{k\in\mathbb{N}}\subset\mathcal{N}^-$ such that, up to a subsequence, $w_k\rightarrow w$ where $w\in\mathcal{N}^-$ satisfies $J(w)=c_{\mathcal{N}^-}$. Furthemore,  we obtain the following inequality
		\begin{equation}\label{i}
			J(v)\geq J(w_k)-\frac{1}{k}\|v-w_k\| \text{ for all } v\in\mathcal{N}^-.	
		\end{equation}
		Let us define the function $f_k(t)=f_k(t\varphi)$ where $f_k(t)(w_k+t\varphi)\in \mathcal{N^-}$ is the function obtained in Lemma \ref{curva}. It is important to mention that the right derivative of zero of this function denoted by $f^\prime_{k}(0)$ is bounded by a positive constant. More specifically, we have that $f^{\prime}_{k}(0)$ exists and it satisfies $|f^{\prime}_{k}(0)|\leq C$ for some constant $C > 0$. Firstly, we shall verify that $A(w_k+t\varphi)-A(w_k) \geq 0$ holds for each $\varphi\geq 0$. In fact, for each $\varphi\geq 0$, we have that
		$$A(w_k+t\varphi)-A(w_k)=\int_{\mathbb{R}^N}(I_{\alpha_1}*a|w_k+t\varphi|^q)a(x)|w_k+t\varphi|^q-(I_{\alpha_1}*a|w_k|^q)a(x)|w_k|^q dx.$$ It is not hard to see that
		\begin{eqnarray}\label{2.14}
			A(w_k+t\varphi)-A(w_k)&=&\int_{\mathbb{R}^N}(I_{\alpha_1}*a|w_k+t\varphi|^q)a(x)[|w_k+t\varphi|^q-|w_k|^q]dx \nonumber \\
			&+&\int_{\mathbb{R}^N}(I_{\alpha_1}*a[|w_k+t\varphi|^q-|w_k|^q])a(x)|w_k|^q dx\geq 0.
		\end{eqnarray}
		Let us assume by contradiction that $f^\prime_{k}(0)$ does not exist. Consider $t_{\theta}\rightarrow 0$ where $t_{\theta}>0$ in such way that $f_k$ satisfies
		$$f^\prime_{k}(0):=\lim_{\theta\rightarrow+\infty}\frac{f_n(t_{\theta})-1}{t_{\theta}}\quad\text{where}\quad f^\prime_{k}(0)\in[-\infty,+\infty].$$
		Notice also that $f_k(0)=1$ and $(w_k)_{k\in\mathbb{N}}\subset\mathcal{N}^-$. Hence, we have that $f_k(t)(w_k+t\varphi)\in\mathcal{N}^-$. As a consequence, we deduce the following identities
		\begin{eqnarray}
			\|w_k\|^2-\lambda A(w_k)-\mu B(w_k)=0=f^2_k(t)\|w_k+t\varphi\|^2-\lambda f^{2q}_k(t)A(w_k+t\varphi)-\mu f^{2p}_k(t) B(w_k+t\varphi).
		\end{eqnarray}
		Therefore, we obtain
		\begin{eqnarray}
			0&=&[f^2_k(t)-1]\|w_k+t\varphi\|^2-\lambda[f^{2q}_k(t)-1] A(w_k+t\varphi)-\mu[f^{2p}_k(t)-1]B(w_k+t\varphi)\nonumber \\
			&+&(\|w_k+t\varphi\|^2-\|w_k\|^2)-\lambda[A(w_k+t\varphi)-A(w_k)]-\mu[B(w_k+t\varphi)-B(w_k)]. \nonumber
		\end{eqnarray}
		Now, by using (\ref{2.14}), we infer that
		\begin{eqnarray}
			0&\leq&[f^2_k(t)-1]\|w_k+t\varphi\|^2-\lambda[f^{2q}_k(t)-1]A(w_k+t\varphi)-\mu[f^{2p}_k(t)-1]B(w_k+t\varphi) \nonumber \\
			&+&(\|w_k+t\varphi\|^2-\|w_k\|^2)-\mu[B(w_k+t\varphi)-B(w_k)]. \nonumber
		\end{eqnarray}
		
		Furthermore, dividing both sides by $t$ and taking the limit as $t\rightarrow0$ , we obtain that
		$$0\leq f^\prime_{k}(0)[2\|w_k\|^2-\lambda 2q A(w_k)-\mu 2pB(w_k)]+2\int_{\mathbb{R}^N}\nabla w_k\nabla\varphi+V(x)w_k\varphi dx-\mu 2pB^{\prime}(w_k)\varphi.$$
		Therefore, using that $(w_k)_{k\in\mathbb{N}}\subset\mathcal{N^-}$, we conclude that $f^\prime_{k}(0)\not=+\infty$.
		Now, let us show that $|f^\prime_{k}(0)|<+\infty$. Suppose by contradiction that $f^\prime_{k}(0)=-\infty$. Then, for all small $t>0$ there holds $f_k(t)<1$. Furthermore, we observe that
		\begin{eqnarray*}
			\|f_k(t)(w_k+t\varphi)-w_k\|=\|[f_k(t)-1]w_k+tf_k(t)\varphi\|\leq[1-f_k(t)]\|w_k\|+tf_k(t)\|\varphi\|.
		\end{eqnarray*}
		It follows from \eqref{i} that
		\begin{eqnarray}
			[1-f_k(t)]\frac{\|w_k\|}{k}+tf(t)\frac{\|\varphi\|}{k}&\geq& J(w_k)-J(f_k(t)(w_k+t\varphi)) \nonumber\\
			&=&\frac{1}{2}\|w_k\|^2-\frac{\lambda}{2q}A(w_k)-\frac{\mu}{2p}B(w_k)-\frac{1}{2}f^2_k(t)\|w_k+t\varphi\|^2\nonumber\\
			&+&\frac{\lambda}{2q}f^{2q}_k(t)A(w_k+t\varphi)+\frac{\mu}{2p}f^{2p}_k(t)B(w_k+t\varphi).\nonumber 
		\end{eqnarray}
		On the other hand, taking into account that $w_k$ and $f_k(t)(w_k+t\varphi)$ belong to $\mathcal{N}^-$, we write the last estimate as follows
		\begin{eqnarray}
			[1-f_k(t)]\frac{\|w_k\|}{k}+tf(t)\frac{\|\varphi\|}{k}&\geq&\left(\frac{1}{2}-\frac{1}{2q}\right)\|w_k\|^2-\left(\frac{1}{2p}-\frac{1}{2q}\right)\mu B(w_k)\nonumber \\
			&+&\left(\frac{1}{2q}-\frac{1}{2}\right)f^2_k(t)\|w_k+t\varphi\|^2+\left(\frac{1}{2p}-\frac{1}{2q}\right)f^{2p}_k(t)B(w_k+t\varphi). \nonumber 
		\end{eqnarray}
		Hence, we obtain that
		\begin{eqnarray}
			[1-f_k(t)]\frac{\|w_k\|}{k}+tf(t)\frac{\|\varphi\|}{k}&\geq&\left(\frac{1}{2q}-\frac{1}{2}\right)[\|w_k+t\varphi\|^2-\|w_k\|^2]+\left(\frac{1}{2q}-\frac{1}{2}\right)[f^2_k(t)-1]\|w_k+t\varphi\|^2\nonumber \\
			&+&\left(\frac{1}{2p}-\frac{1}{2q}\right)[f^{2p}_k(t)-1]\mu B(w_k) 
			+\left(\frac{1}{2p}-\frac{1}{2q}\right)\mu f^{2p}_k(t)[B(w_k+t\varphi)-B(w_k)]. \nonumber
		\end{eqnarray}
		Therefore, by dividing both sides by $t$ and taking the limit as $t\rightarrow 0$, we deduce that
		\begin{eqnarray}\label{2.16}
			\frac{\|\varphi\|}{k}&\geq&\frac{1}{q}\left[(1-q)\|w_k\|^2-(p-q)\mu B(w_k)+\frac{q\|w_k\|}{k}\right] f^\prime_{k}(0)\nonumber \\
			&+&\left(\frac{1-q}{q}\right)\int_{\mathbb{R}^N}\nabla w_k\nabla\varphi+V(x)w_k\varphi dx-\left(\frac{p-q}{q}\right)\mu B^\prime(w_k)\varphi.
		\end{eqnarray}
		Recall that $(w_k)_{k\in\mathbb{N}}\subset\mathcal{N}^-$. Therefore,  we have that $(2-2q)\|w_k\|^2-(2p-2q)\mu B(w_k)<0$. Now, by using the fact that $w_k\rightarrow w$ and $c_{\mathcal{N}^-}=J(w)$, there exists $C_4 > 0$ such that $(2-2q)\|w_k\|^2-(2p-2q)\mu B(w_k)< - C_4$. 
		As a consequence, for $k$ sufficiently large, we infer that 
		$$(1-q)\|w_k\|^2-(p-q)\mu B(w_k)+\frac{\|w_k\|q}{k}\leq -C_4+\frac{C_3 q}{k}<0.$$
		Thus, we can write inequality \eqref{2.16} as follows
		\begin{eqnarray*}
			\frac{\frac{\|\varphi\|}{k}-\left(\frac{1-q}{q}\right)\displaystyle\int_{\mathbb{R}^N}\nabla w_k\nabla\varphi+V(x)w_k\varphi dx+\left(\frac{p-q}{q}\right)\mu B^\prime(w_k)\varphi}{\frac{1}{q}\left[(1-q)\|w_k\|^2-(p-q)\mu B(w_k)+\frac{q\|w_k\|}{k}\right]}\leq f^\prime_{k}(0).
		\end{eqnarray*}
		Therefore, we can establish a lower bound for $f^\prime_{k}(0)$. This is a contradiction due to the fact  that $f^\prime_{k}(0)=-\infty$. Hence, $|f^\prime_{k}(0)|<+\infty$ is now satisfied. Based on the estimates given above we obtain that $|f^\prime_{k}(0)|\leq C_5$ is verified for some constant $C_5 > 0$.
		
		In what follows we shall prove that the function $w\in\mathcal{N}^-$ for which $J(w)=c_{\mathcal{N}^-}$ is a weak solution to problem \eqref{p1}. As a first step, by using the inequality \eqref{i}, we can apply the Ekeland’s Variational Principle for $f_k(t)(w_k+t\varphi)$ and $(w_k)_{k\in\mathbb{N}}$ which implies that
		\begin{eqnarray}
			\frac{1}{k}[|f_k(t)-1|\|w_k\|+tf_k(t)\|\varphi\|]&\geq&\frac{1}{k}\|f_k(t)(w_k+t\varphi)-w_k\| \nonumber \\  &\geq& J(w_k)-J(f_k(t)(w_k+t\varphi))=\frac{1}{2}\|w_k\|^2-\frac{\lambda}{2q}A(w_k)-\frac{\mu}{2p}B(w_k)\nonumber \\
			&-&\frac{f^2_k(t)}{2}\|w_k+t\varphi\|^2 +\frac{\lambda f^{2q}_k(t)}{2q}A(w_k+t\varphi)+\frac{\mu f^{2p}_k(t)}{2p}B(w_k+t\varphi). \nonumber
		\end{eqnarray}
		Once more, the last estimate implies also that
		\begin{eqnarray}	
			\frac{1}{k}[|f_k(t)-1|\|w_k\|+tf_k(t)\|\varphi\|]&\geq&-\left[\frac{f^2_k(t)-1}{2}\right]\|w_k\|^2+\lambda\left[\frac{f^{2q}_k(t)-1}{2q}\right]A(w_k+t\varphi)\nonumber \\
			&+&\mu\left[\frac{f^{2p}_k(t)-1}{2p}\right]B(w_k+t\varphi)
			+\frac{f^2_k(t)}{2}(\|w_k\|^2-\|w_k+t\varphi\|^2)\nonumber \\
			&+&\frac{\lambda}{2q}[A(w_k+t\varphi)-A(w_k)]+\frac{\mu}{2p}[B(w_k+t\varphi)-B(w_k)]\nonumber
		\end{eqnarray}
		holds for each $\varphi\geq 0$. Therefore, dividing both sides by $t$ and taking the limit as $t\rightarrow 0$, we deduce that 
		\begin{eqnarray}
			\frac{1}{k}(|f^\prime_{k}(0)|\|w_k\|+\|\varphi\|)&\geq&-f^\prime_{k}(0)[\|w_k\|^2-\lambda A(w_k)-\mu B(w_k)]-\int_{\mathbb{R}^N}\nabla w_k\nabla\varphi+V(x)w_k\varphi dx+\mu B^\prime(w_k)\varphi\nonumber \\ 
			&+&\liminf_{t\rightarrow 0^+}\frac{\lambda}{2q}\left[\frac{A(w_k+t\varphi)-A(w_k)}{t}\right]  \\
			&=&-\int_{\mathbb{R}^N}\nabla w_k\nabla\varphi+V(x)w_k\varphi dx+\mu B^\prime(w_k)\varphi+\liminf_{t\rightarrow 0^+}\frac{\lambda}{2q}\left[\frac{A(w_k+t\varphi)-A(w_k)}{t}\right].\nonumber
		\end{eqnarray}
		According to \eqref{2.14} we observe that $A(w_k+t\varphi)-A(w_k)\geq 0$ holds for each $\varphi \geq 0$. Now, by applying Fatou's Lemma together with Mean Value Theorem, we deduce that
		\begin{eqnarray}\label{fatou}
			\liminf_{t\rightarrow 0^+}\frac{\lambda}{2q}\left[\frac{A(w_k+t\varphi)-A(w_k)}{t}\right]\geq\lambda\int_{\mathbb{R}^N}(I_{\alpha_1}*a|w_k|^q)a(x)|w_k|^{q-2}w_k\varphi dx=\lambda A^\prime(w_k)\varphi.
		\end{eqnarray}
		It follows also that
		$$\frac{1}{k}(|f^\prime_{k}(0)|\|w_k\|+\|\varphi\|)+\int_{\mathbb{R}^N}\nabla w_k\nabla\varphi+V(x)w_k\varphi dx-\mu B^\prime(w_k)\varphi\geq \lambda A^\prime(w_k)\varphi.$$
		Therefore, using the fact that $|f^\prime_{k}(0)|\leq C_5$ and $\|w_k\|\leq C_3$, we obtain
		$$\frac{C+\|\varphi\|}{k}+\int_{\mathbb{R}^N}\nabla w_k\nabla\varphi+V(x)w_k\varphi dx-\mu B^\prime(w_k)\varphi\geq \lambda A^\prime(w_k)\varphi.$$
		Now, doing $k \to +\infty$ and taking into account that $w_k\rightarrow w$, we can use the Fatou's Lemma once more proving that
		\begin{equation}\label{s1}
			\int_{\mathbb{R}^N}\nabla w\nabla\varphi+V(x)w\varphi dx-\mu B^\prime(w)\varphi-\lambda A^\prime(w)\varphi\geq 0 \, \,\text{ for each } \,\,\varphi\geq 0.
		\end{equation}
		
		At this stage we shall prove that $w$ is a weak solution for any $\varphi\in X$. In order to do that we consider $\Psi=(w+\epsilon\phi)^+$ where $\phi\in X$. Now, by using $\Psi$ as testing function in \eqref{s1}, we obtain
		\begin{eqnarray}
			0&\leq&\int_{\mathbb{R}^N}\nabla w\nabla\Psi+V(x)w\Psi dx-\mu B^\prime(w)\Psi-\lambda A^\prime(w)\Psi\nonumber \\
			&=&\int_{[w+\epsilon\phi\geq 0]}\nabla w\nabla(w+\epsilon\phi)+V(x)w(w+\epsilon\phi)dx \nonumber \\
			&-&\int_{[w+\epsilon\phi\geq 0]}\mu(I_{\alpha_2}*|w|^p)|w|^{p-2}w(w+\epsilon\phi)+\lambda(I_{\alpha_1}*a|w|^q)a(x)|w|^{q-2}w(w+\epsilon\phi) dx \nonumber \\	
			&=&\left(\int_{\mathbb{R}^N}-\int_{[w+\epsilon\phi< 0]}\right)\nabla w\nabla(w+\epsilon\phi)+V(x)w(w+\epsilon\phi)-\mu(I_{\alpha_2}*|w|^p)|w|^{p-2}w(w+\epsilon\phi) \nonumber \\
			&-&\lambda(I_{\alpha_1}*a|w|^q)a(x)|w|^{q-2}w(w+\epsilon\phi) dx.
		\end{eqnarray}
		The last expression can be rewritten in the following form:
		\begin{eqnarray}
			0&=&\|w\|^2-\lambda A(w)-\mu B(w)+\epsilon\int_{\mathbb{R}^N}\nabla w\nabla\phi+V(x)w\phi dx-\mu\epsilon B^\prime(w)\phi-\lambda\epsilon A^\prime(w)\phi dx \\
			&-&\int_{[w+\epsilon\phi< 0]}\nabla w\nabla(w+\epsilon\phi)+V(x)w(w+\epsilon\phi)dx. \nonumber \\
			&-& \int_{[w+\epsilon\phi< 0]}\mu(I_{\alpha_2}*|w|^p)|w|^{p-2}w(w+\epsilon\phi)+\lambda(I_{\alpha_1}*a|w|^q)a(x)|w|^{q-2}w(w+\epsilon\phi) dx. \nonumber
		\end{eqnarray}
		Hence, we deduce that
		\begin{eqnarray}\label{2.23}
			0&\leq&\epsilon\left[\int_{\mathbb{R}^N}\nabla w\nabla\phi+V(x)w\phi-\mu B^\prime(w)\phi-\lambda A^\prime(w)\phi dx\right]-\epsilon\int_{[w+\epsilon\phi<0]}\nabla w\nabla\phi+V(x)w\phi dx.
		\end{eqnarray}
		Now, by applying the Dominated Convergence Theorem, we see that $$\int_{[w+\epsilon\phi<0]}\nabla w\nabla\phi+V(x)w\phi dx\rightarrow 0$$ as $\epsilon\rightarrow 0$. In fact, by using the fact that $J(u)=J(|u|)$, we assume without loss of generality that $w\geq 0$. Hence, $\chi_{[w+\epsilon\varphi<0]}\rightarrow 0$ a.e. in $\mathbb{R}^N$ as $\epsilon\rightarrow 0$. Furthermore, we mention that
		$$\int_{[w+\epsilon\phi<0]}\nabla w\nabla\phi+V(x)w\phi dx=\int_{\mathbb{R}^N}\chi_{[w+\epsilon\varphi<0]}(\nabla w\nabla\phi+V(x)w\phi) dx \leq \int_{\mathbb{R}^N}(\nabla w\nabla\phi+V(x)w\phi) dx<+\infty.$$ 
		Therefore, the desired result follows by applying the Dominated Convergence Theorem. Furthermore, by dividing expression \eqref{2.23} above by $\epsilon$ and taking the limit as $\epsilon \rightarrow 0$, we obtain
		\begin{eqnarray}
			\int_{\mathbb{R}^N}\nabla w\nabla\phi+V(x)w\phi dx-\mu B^\prime(w)\phi-\lambda A^\prime(w)\phi\geq 0
		\end{eqnarray}
		for any $\phi\in X$. Therefore, by taking the direction $-\phi$, we can conclude that $w$ is a weak solution for problem \eqref{p1}. This ends the proof.
	\end{proof}
	
	\begin{lem}\label{2.8}
		Suppose $(h_1)-(h_3)$ holds. Assume also that $\lambda\in(0,\lambda^*)$ holds. Then $c_{\mathcal{N}^+}=J(v)<0$.
	\end{lem}
	\begin{proof} Let $\lambda\in(0,\lambda^*)$ be fixed. It is worth noting that $R_e(tu)<R_n(tu)$ for each $t\in(0,t_e(u))$, see Remark \ref{2.4} and Figure \ref{figure2}. It is important to emphasize that $R_e(tu)=R_n(tu)$ if, and only if $t=t_e(u)$. Now, we claim that $t^{n,+}(u) < t_e(u)$ holds. In fact, we observe that $t_e(u)=p^{\frac{1}{2p-2}}t_n(u)$ for each $p\in\left(2_{\alpha_2},2^*_{\alpha_2}\right)$. Recall also that $p^{\frac{1}{2p-2}}>1$. Hence, we obtain that $t^{n,+}(u)<t_e(u)$ and
		$t^{n,+}(u)<t_n(u)<t_e(u).$ Under these conditions, we observe that $R_e(t^{n,+}(u)u)<R_n(t^{n,+}(u)u)=\lambda$. It follows from Remark \ref{2o1} that $R_e(t^{n,+}(u)u)<\lambda$ if and only if $J(t^{n,+}(u)u)<0$. Therefore, by using the fact that $t^{n,+}(u)u\in\mathcal{N}^+$, we infer that
		$$c_{\mathcal{N}^+}=\inf_{v\in\mathcal{N}^+}J(v)\leq J(t^{n,+}(u)u)<0.$$
		This finishes the proof.
	\end{proof} 
	
	\begin{lem}\label{2.9}
		Suppose $(h_1)-(h_3)$ and $\lambda\in(0,\lambda^*)$ holds. Consider $(u_k)_{k\in\mathbb{N}}\subset\mathcal{N}^+$ as a minimizing sequence for the functional $J$ in $\mathcal{N}^+$. Then there exists $u\in X\setminus\{0\}$ such that, up to a subsequence, $u_k\rightarrow u$ in $X$ where $u\in\mathcal{N}^+$. Furthermore, we obtain that $c_{\mathcal{N}^+}=J(u)$.
	\end{lem}
	\begin{proof}	As was mentioned in Lemma \ref{coercive} the energy functional $J$ is coercive over $\mathcal{N}$. Consider  $(u_k)_{k\in\mathbb{N}}\subset\mathcal{N}^+$ as a minimizing sequence. Therefore, we obtain that $u_k\rightharpoonup u$ in $X$ for some $u \in X$. Moreover, using the fact that 
		the embedding $X\hookrightarrow L^r(\mathbb{R}^N),  r\in[2,2^*)$ is compact, we deduce that $u_k\rightarrow u$ in $L^r(\mathbb{R}^N)$, $u_k\rightarrow u$ a.e. in $\mathbb{R}^N$ and $|u_k|\leq h_r\in L^r(\mathbb{R}^N)$. Notice also that	
		$$\lambda A(u_k)=\left(\frac{2q}{1-q}\right)\left(\frac{2(p-1)}{4p}\right)B(u_k)-\left(\frac{2q}{1-q}\right)J(u_k). $$
		Now, we we assume that $u_k\rightharpoonup 0$. Using the compact embedding listed just above together with the Dominated Convergence Theorem, we know that $A(v_k)\rightarrow 0$ and $B(u_k)\rightarrow 0$. Furthermore, we mention also that  $J(u_k)\rightarrow c_{\mathcal{N}^+}$. In particular, we obtain
		\begin{eqnarray}
			0 =	\lim_{k\rightarrow+\infty} \lambda A(u_k)&=&\lim_{k\rightarrow+\infty}\left(\frac{2q}{1-q}\right)\left(\frac{2(p-1)}{4p}\right)B(u_k)-\left(\frac{2q}{1-q}\right)J(u_k) = -\left(\frac{2q}{1-q}\right)c_{\mathcal{N}^+} > 0. 
		\end{eqnarray} 
		Here was used the fact that $c_{\mathcal{N}^+}<0$, see  Lemma \ref{2.8}.
		This is a contradiction proving that $u_k\rightharpoonup u$ where $u\neq 0$. Furthermore, we infer that 
		$$\lambda A(u)\geq-\frac{2q}{1-q}c_{\mathcal{N}^+}>0.$$
		Now, we assume by contradiction that $u_k\not\rightarrow u$. Under these conditions, we obtain $\|u\|<\liminf_{k\rightarrow +\infty}\|u_k\|$.
		Hence, by using Proposition \ref{2.3}, there exists a unique $t^{n,+}(u)>0$ such that $t^{n,+}(u)u\in\mathcal{N}^+$. Furthermore, we know that $\phi^\prime(t^{n,+}(u))=J^\prime(t^{n,+}(u)u)u=0$ and $\phi(t^{n,+}(u))=J(t^{n,+}(u)u)<0$. As a consequence, using once again the compact embedding $X\hookrightarrow L^r(\mathbb{R}^N), r\in[2,2^*)$, we obtain that
		$$\frac{d}{dt}J(tu)=J^\prime(tu)u<\liminf_{k\rightarrow +\infty}J^\prime(tu_k)u_k=\liminf_{k\rightarrow +\infty}\frac{d}{dt}J(tu_k)\leq 0$$
		holds for each $t\in(0,1]$. The last statement implies that $t^{n,+}(u)>1$. Hence, we obtain that
		$$\frac{d}{dt}J(tu_k)=J^\prime(tu_k)u_k\leq 0, \,\,\, t\in[0,1].$$
		Under these conditions, we infer that  
		$$c_{\mathcal{N}^+}\leq J(t^{n,+}(u)u)\leq J(u)<\liminf J(u_k)=c_{\mathcal{N}^+}.$$ This is a contradiction proving that $u_k\rightarrow u$ in $X$. Since the functional $J$ is continuous the desired result follows.
	\end{proof}
	\begin{prop}\label{p2.8}
		Suppose $(h_1)-(h_3)$ holds. Assume also that $\lambda\in(0,\lambda^*)$. Then there exists a function $u\in\mathcal{N}^+$ obtained in the same way of Lemma \ref{2.9} in such way that $c_{\mathcal{N}^+}=J(u)$. Furthermore, $u$ is a weak solution for the problem \eqref{p1}.
	\end{prop}
	\begin{proof}
		According to Lemma \ref{2.9} we obtain a function $u \in \mathcal{N}^+$ such that $c_{\mathcal{N}^+} = J(u)$. On the other hand, the manifold $\mathcal{N}^+$ is not closed set. Under these conditions, we infer that $\overline{\mathcal{N}^+} = \mathcal{N}^+ \cup \{0\}$ which is a closed set. Hence, we can apply the Ekeland’s
		Variational Principle ensuring the existence of a minimizing sequence which we shall still denote as $(u_k)_{k \in \mathbb{N}} \subset \mathcal{N}^+$. Up to a subsequence, we have $u_k \rightarrow u$ where $u \in \mathcal{N}^+$ satisfies $J(u) = c_{\mathcal{N}^+}$. Furthermore, we infer that
		\begin{equation}\label{ii}
			J(w)\geq J(u_k)-\frac{1}{k}\|w-u_k\| \text{ for all } w \in\mathcal{N}^+.
		\end{equation}
		Define the function $h_k(t) = h_k(t\varphi)$ where $h_k(t)$ is the function obtained in Lemma \ref{curva2}. Once again we shall verify that the right derivative at zero of $h_k$, denoted by $h^\prime_k(0)$, is bounded by a positive constant, that is, $h^\prime_k(0)$ exists and satisfies $|h^\prime_k(0)| \leq C$ for some constant $C > 0$. Now, we assume by contradiction that $h^\prime_k(0) = +\infty$. Hence, for small values of $t$, we obtain that $h_k(t) > 1$. Hence, 
		\begin{equation}
			\|h_k(t)(u_k+t\varphi)-u_k\|=\|(h_k(t)-1)u_k+th_k(t)\varphi\|\leq|h_k-1|\|u_k\|+th_k(t)\|\varphi\|.
		\end{equation}
		Therefore, by using \eqref{ii}, we see that
		\begin{equation}\label{2.35}
			[h_k(t)-1]\frac{\|u_k\|}{k}+th_k(t)\frac{\|\varphi\|}{k}\geq J(u_k)-J(h_k(t)(u_k+t\varphi)).
		\end{equation}
		Recall also that $u_k \in \mathcal{N}^+$ and $h_k(t)(u_k + t\varphi) \in \mathcal{N}^+$. Under these conditions, the last estimate can be written as follows:
		\begin{equation*}
			J(u_k)=\left(\frac{1}{2}-\frac{1}{2q}\right)\|u_k\|^2-\left(\frac{1}{2p}-\frac{1}{2q}\right)\mu B(u_k).
		\end{equation*}
		As a consequence, by using \eqref{2.35}, we infer that 
		\begin{eqnarray*}
			[h_k(t)-1]\frac{\|u_k\|}{k}+th_k(t)\frac{\|\varphi\|}{k}&\geq&\left(\frac{1}{2}-\frac{1}{2q}\right)\|u_k\|^2-\left(\frac{1}{2p}-\frac{1}{2q}\right)\mu B(u_k) \nonumber \\ &-&\left(\frac{1}{2}-\frac{1}{2q}\right)h^2_k(t)\|u_k+t\varphi\|^2+\left(\frac{1}{2p}-\frac{1}{2q}\right)\mu h^{2p}_k(t)B(u_k+t\varphi).
		\end{eqnarray*}
		Therefore, we obtain that 
		\begin{eqnarray}
			[h_k(t)-1]\frac{\|u_k\|}{k}+th_k(t)\frac{\|\varphi\|}{k}\geq -\left(\frac{1}{2}-\frac{1}{2q}\right)[\|u_k+t\varphi\|^2-\|u_k\|^2]-\left(\frac{1}{2}-\frac{1}{2q}\right)[h^2_k(t)-1]\|u_k+t\varphi\|^2 \nonumber \\+\left(\frac{1}{2p}-\frac{1}{2q}\right)\mu[h_k^{2p}(t)-1]B(u_k)+\left(\frac{1}{2p}-\frac{1}{2q}\right)\mu h_k^{2p}(t)[B(u_k-t\varphi)-B(u_k)]. \nonumber
		\end{eqnarray}
		Dividing both sides of the inequality by $t$ and taking the limit as $t\rightarrow 0$, we see that
		\begin{eqnarray}\label{cota}
			\frac{\|\varphi\|}{k}\geq\frac{1}{q}\left[(1-q)\|u_k\|^2-(p-q)\mu B(u_k)-q\frac{\|u_k\|^2}{k}\right]h^\prime_k(0) \nonumber \\
			+\left(\frac{1-q}{q}\right)\int_{\mathbb{R}^N}\nabla u_k\nabla\varphi+V(x)u_k\varphi dx-\frac{p-q}{q}\mu B^\prime(u_k)\varphi.
		\end{eqnarray}
		Let us analyze the sign of the expression within the parentheses given just above. Since $(u_k)_{k\in\mathbb{N}}\subset\mathcal{N}^+$ we mention that $(1-q)\|u_k\|^2-(p-q)\mu B(u_k)>0$. According to Lemma \ref{2.9} there exists an element $u\in\mathcal{N}^+$ such that $u_k \to u$ in $X$ and $J(u)=c_{\mathcal{N}^+}$. Therefore, we find $k_0\in\mathbb{N}$ such that for all $k\geq k_0$ there holds 
		\begin{eqnarray*}
			(1-q)\|u_k\|^2-(p-q)\mu B(u_k)\geq(1-q)\|u\|^2-(1-q)\mu B(u)-\frac{1}{k}\geq C_1-\frac{1}{k}\geq C_2>0. 	
		\end{eqnarray*}
		Hence, there exists a constant $C_3$ such that
		\begin{equation*}
			(1-q)\|u_k\|^2-(p-q)\mu B(u_k)\geq C_3>0 \quad \text{for all}\ k\in\mathbb{N}.
		\end{equation*}
		Since $(u_k)_{k\in\mathbb{N}}$ as a minimizing sequence there exists a constant $C_4>0$ such that $\|u_k\|\leq C_4$. Hence, for all sufficiently large $k$, we infer that
		$$(1-q)\|u_k\|^2-(p-q)\mu B(u_k)-q\frac{\|u_k\|}{k}\geq C_3-q\frac{C_4}{k}>0.$$
		Under these conditions, by using \eqref{cota}, we see that
		$$\frac{-\left(\frac{1-q}{q}\right)\displaystyle\int_{\mathbb{R}^N}\nabla u_k\nabla\varphi+V(x)u_k\varphi dx+\left(\frac{p-q}{q}\right)\mu B^\prime(u_k)\varphi}{\frac{1}{q}\left[(1-q)\|u_k\|^2-(p-q)\mu B(u)-q\frac{\|u_k\|}{k}\right]}\geq h^\prime_k(0).$$
		The last estimate implies that there exists an upper bound for $h^\prime_k(0)$. This is a contradiction due to the fact that $h^\prime_k(0)=+\infty$.
		
		Now, we shall prove that $h^\prime_k(0)\neq-\infty$. In fact,  by using the fact that  $h_k(0)=1$ and $h_k(t)(u_k+t\varphi)\in\mathcal{N}^+$ together with $(u_k)_{k\in\mathbb{N}}\subset\mathcal{N}^+$, we obtain
		\begin{eqnarray*}
			\|u_k\|^2-\lambda A(u_k)-\mu B(u_k)=0=h^2_k(t)\|u_k+t\varphi\|^2-\lambda h^{2q}_k(t)A(u_k+t\varphi)-\mu h^{2p}_k(t)B(u_k+t\varphi).
		\end{eqnarray*}
		Using the last assertion implies that
		\begin{eqnarray*}
			0=[h^2_k(t)-1]\|u_k+t\varphi\|^2+(\|u_k+t\varphi\|^2-\|u_k\|^2)-\lambda[h^{2q}_k(t)-1]A(u_k+t\varphi) \nonumber \\
			-\lambda[A(u_k+t\varphi)-A(u_k)]-\mu[h^{2p}_k(t)-1]B(u_k+t\varphi)-\mu[B(u_k+t\varphi)-B(u_k)].
		\end{eqnarray*}
		Since we already know that $A(u_k+t\varphi)-A(u_k)\geq 0$ holds for all $\varphi\geq 0$, we infer that
		\begin{eqnarray*}
			0\leq[h^2_k(t)-1]\|u_k+t\varphi\|^2+(\|u_k+t\varphi\|^2-\|u_k\|^2)-\lambda[h^{2q}_k(t)-1]A(u_k+t\varphi) \nonumber \\
			-\mu[h^{2p}_k(t)-1]B(u_k+t\varphi)-\mu[B(u_k+t\varphi)-B(u_k)].
		\end{eqnarray*}
		Dividing both sides by $t$ in the last estimate and taking the limit as $t\rightarrow 0^+$, we obtain
		\begin{equation*}
			0\leq[2\|u_k\|^2-\lambda 2qA(u_k)-\mu 2pB(u_k)]h^\prime_k(0)+2\langle u_k,\varphi\rangle-\mu2p B^\prime(u_k)\varphi.
		\end{equation*}
		Hence, we see that 
		\begin{equation}
			\frac{-2\langle u_k,\varphi\rangle+\mu2pB^\prime(u_k)\varphi}{2\|u_k\|^2-\lambda2qA(u_k)-\mu 2pB(u_k)}\leq h^\prime_k(0).
		\end{equation}
		The last estimate give us a lower bound for $h^\prime_k(0)$. Therefore, $h^\prime_k(0)\neq -\infty$. As a consequence, we deduce that $|h^\prime_k(0)|\leq C_5$. 
		
		At this stage, we shall prove that $u$ is a weak solution to problem \eqref{p1}. Firstly, by using the condition \eqref{ii}, we infer that 
		\begin{equation}
			\frac{1}{k}[|h_k(t)-1|\|u_k\|+th_k(t)\|\varphi\|]\geq\frac{1}{k}\|h_k(t)(u_k+t\varphi)-u_k\|\geq J(u_k)-J(h_k(t)(u_k+t\varphi)).
		\end{equation}
		Hence, for all $\varphi\in X$ with $\varphi\geq 0$, there holds
		\begin{eqnarray}
			\frac{1}{k}[|h_k(t)-1|\|u_k\|+th_k(t)\|\varphi\|]&\geq&-\left[\frac{h^2_k(t)-1}{2}\right]\|u_k\|^2-\frac{h^2_k(t)}{2}[\|u_k+t\varphi\|^2-\|u_k\|^2] \nonumber \\
			&+&\lambda\left[\frac{h^{2q}_k(t)-1}{2q}\right]A(u_k+t\varphi)+\frac{\lambda}{2q}[A(u_k+t\varphi)-A(u_k)]\nonumber \\ &+&\mu\left[\frac{h^{2p}_k(t)-1}{2p}\right]B(u_k+t\varphi)+\frac{\mu}{2p}[B(u_k+t\varphi)-B(u_k)].
		\end{eqnarray}
		Now, using the last estimate and dividing both sides by $t$ and taking the limit as $t\rightarrow 0^+$, we infer that 
		\begin{eqnarray}
			\frac{1}{k}h^\prime_k(0)\|u_k\|+\frac{\|\varphi\|}{k}&\geq&-h^\prime_k(0)[\|u_k\|^2-\lambda A(u_k)-\mu B(u_k)]-\langle u_k,\varphi\rangle+\mu B^\prime(u_k)\varphi
			+\liminf_{t\rightarrow 0^+}\frac{\lambda}{2q}\frac{A(u_k+t\varphi)-A(u_k)}{t} \nonumber \\
			&=&-\langle u_k,\varphi\rangle+\mu B^\prime(u_k)\varphi+\liminf_{t\rightarrow 0^+}\frac{\lambda}{2q}\frac{A(u_k+t\varphi)-A(u_k)}{t}. \nonumber \\
		\end{eqnarray}
		On the other hand, by using Fatou's Lemma, we obtain that
		$$\liminf_{t\rightarrow 0^+}\frac{\lambda}{2q}\frac{A(u_k+t\varphi)-A(u_k)}{t}\geq\lambda A^\prime(u)\varphi.$$ 
		Hence, the last estimate implies that 
		\begin{equation}
			\frac{h^\prime_k(0)\|u_k\|+\|\varphi\|}{k}+\langle u_k,\varphi\rangle-\mu B^\prime(u_k)\varphi\geq\lambda A^\prime(u)\varphi.
		\end{equation} 
		Recall also that $u_k \rightarrow u$ in $X$. Applying the Fatou's Lemma and the Dominated Convergence Theorem, we see that 
		\begin{equation}
			\int_{\mathbb{R}^N}\nabla u\nabla\varphi+V(x)u\varphi dx-\mu B^\prime(u)\varphi-\lambda A^\prime(u)\varphi\geq0\quad\text{for all}\ \varphi\geq0.
		\end{equation}
		Furthermore, by using the same argument as was done in the proof of Proposition \ref{weak1}, we conclude that $u$ is a weak solution to problem \eqref{p1}.
		This ends the proof.
	\end{proof}
	It is important to emphasize that the set $\mathcal{N}^0$ is empty for $\lambda \in (0, \lambda^*)$, see Proposition \ref{2.55}.  However, for $\lambda = \lambda^*$, we obtain that $\mathcal{N}^0$ is nonempty set. In what follows we need to prove  that does not exist weak solutions $u$ to problem \eqref{p1} in such way that $u \in \mathcal{N}^0$ for $\lambda = \lambda^*$. This is a fundamental tool in order to prove our main results for $\lambda = \lambda^*$. Here the main idea is to prove a nonexistence result by using some arguments employed in \cite{kaye}.  As a first step, we consider the following result:
	\begin{lem}\label{2.10}
		Suppose $(h_1)-(h_3)$ holds. Assume also $\Lambda_n(u) = \lambda^*$. Then we obtain that 
		$$2\langle u,\varphi\rangle-2p\mu B^\prime(u)\varphi-2q\lambda^* A^\prime(u)\varphi=0, u \in \mathcal{N}^0, \,\, \mbox{for all} \,\, \varphi\in X.$$
	\end{lem}
	\begin{proof}
		According to Lemma \ref{lema2.1}, we observe that $$\Lambda_n(u)=C_{p,q,\mu}\frac{\|u\|^{2\left(\frac{p-q}{p-1}\right)}}{B(u)^{\frac{1-q}{p-1}}A(u)}, u \in X \setminus \{0\}.$$
		Consider the functions $f: X \setminus \{0\} \to \mathbb{R}$ and $g: X \setminus \{0\}  \to \mathbb{R}$ in the following form:
		$$f(u)=\frac{1}{A(u)}\text{ and }g(u)=\frac{\|u\|^{2\left(\frac{p-q}{p-1}\right)}}{B(u)^{\frac{1-q}{p-1}}}.$$Hence, we write $\Lambda_n(u) = C_{p,q,\mu} f(u) g(u)$.
		
		1) Let us show that $\langle f^\prime(u),\varphi\rangle$ exists for all $\varphi\in X_+$ and for all $u\in\mathcal{N}^0$, where $X_{+} = \{\phi \in X: \phi \geq 0\}$.
		We know that $g(u+t\varphi)$ is well-defined. It is easy to verify that $g^\prime(u)\varphi$ exists for each $\phi \in X$. Recall also that $u$ is a minimum of $\Lambda_n$, that is, we have that $\Lambda_n(u)=\lambda^*$. Therefore, we obtain that 
		$$\Lambda_n(u+t\varphi)-\Lambda_n(u)=\Lambda_n(u+t\varphi)-\lambda^*\geq 0.$$
		The last estimate implies that $f(u+t\varphi)g(u+t\varphi)-f(u)g(u)\geq 0$. It is straightforward to verify that
		\begin{equation}\label{2.43}
			f(u+t\varphi)(g(u+t\varphi)-g(u))\geq-g(u)(f(u+t\varphi)-f(u)).
		\end{equation}
		Define the function $L(t)=f(u+t\varphi)=\frac{1}{A(u+t\varphi)}$. Now, by using the Mean Value Theorem, we infer that
		$$\frac{L(t)-L(0)}{t}=L^\prime(\theta)\quad \theta\in[0,t].$$ 
		Under these conditions, by using the fact that $\theta\rightarrow 0$ as $t\rightarrow 0$, we deduce that
		$$\lim_{t\rightarrow 0}\frac{L(t)-L(0)}{t}=\lim_{t\rightarrow 0}L^\prime(\theta)=L(0).$$
		Furthermore, we mention that
		$$L^\prime(t)=-[A(u+t\varphi)]^{-2}\frac{d}{dt}A(u+t\varphi)=-[A(u+t\varphi)]^{-2}\left[2q\int_{\mathbb{R}^N}(I_{\alpha_1}*a|u+t\varphi|^q)a(x)|u+t\varphi|^{q-2}(u+t\varphi)\varphi dx\right].$$
		In particular, we obtain
		$$L^\prime(0)=-[A(u)]^{-2}\left[2q\int_{\mathbb{R}^N}(I_{\alpha_1}*a|u|^q)a(x)|u|^{q-2}u\varphi dx\right].$$ 
		Therefore, by dividing expression \eqref{2.43} by $t$ and taking the limit as $t\rightarrow 0$, we obtain 
		\begin{equation}
			+\infty>\langle g^\prime(u),\varphi\rangle f(u)\geq -g(u)\liminf_{t\rightarrow 0^+} L^\prime(\theta)=g(u)L^\prime(0).
		\end{equation}
		The last assertion implies that
		\begin{equation}\label{u>0}
			+\infty>\langle g^\prime(u),\varphi\rangle f(u)\geq g(u) [A(u)]^{-2}2q A^\prime(u)\varphi.
		\end{equation}
		Therefore, $0 < A^\prime(u)\varphi < +\infty$ for all $\varphi \in X_+$. As a consequence, we know that $A^\prime(u)\varphi$ exists for each $\phi \geq 0$. Now, by using the last sentence, we obtain that $u>0$ a.e. in $\mathbb{R}^N$. 
		Recall also that $f(u) = [A(u)]^{-1}$. Hence,
		$$\langle f^\prime(u),\varphi\rangle=-2q[A(u)]^{-2}A^\prime(u)\varphi, \,\, \varphi \geq 0.$$
		
		2) Let us prove that $2\langle u,\varphi\rangle-2p B'(u)\varphi-2q\lambda^*A'(u)\varphi\geq 0$ for all $\varphi\in X_+$.
		Assume without loss of generality that $\|u\|=1$. The last sentence is possible due to the fact that $\Lambda_n$ is $0$-homogeneous. Now, we shall calculate the derivative of  $\Lambda_n$ in the direction $\varphi\in X_+$. After some manipulations we obtain that 
		\begin{eqnarray}\label{2.46}
			\frac{\left[2\left(\frac{p-q}{p-1}\right)\langle u,\varphi\rangle B(u)^{\frac{1-q}{p-1}}-2p\left(\frac{1-q}{p-1}\right)B(u)^{\frac{1-q}{p-1}-1}B^\prime(u)\varphi\right] [A(u)]^{-1}}{B(u)^{2\left(\frac{1-q}{p-1}\right)}}-\frac{2q[A(u)]^{-2}A^\prime(u)\varphi}{B(u)^{\frac{1-q}{p-1}}}\geq 0.
		\end{eqnarray}
		Since $u\in\mathcal{N}^0$ we infer that
		\begin{equation*}
			(2-2q)\|u\|^2-(2p-2q)\mu B(u)=0, \quad \ (2-2p)\|u\|^2+(2p-2q)\lambda^* A(u)=0.
		\end{equation*}
		Recall also that $\|u\|=1$. In particular, we deduce that
		\begin{equation}\label{2.85}
			B(u)=\frac{1-q}{(p-q)\mu} \ ,\ A(u)=\frac{p-1}{(p-q)\lambda^*}.
		\end{equation}
		As a consequence, by using \eqref{2.85} and \eqref{2.46}, we mention that
		\begin{equation*}
			\left[\frac{\lambda^*\left(\frac{p-q}{p-1}\right)^2}{B(u)^{\frac{1-q}{p-1}}}\right]2\langle u,\varphi\rangle-\left[\frac{\lambda^*\left(\frac{p-q}{p-1}\right)^2}{B(u)^{\frac{1-q}{p-1}}}\right] 2p\mu B^\prime(u)\varphi-\left[\frac{(\lambda^*)^2\left(\frac{p-q}{p-1}\right)^2}{B(u)^{\frac{1-q}{p-1}}}\right]2q A^\prime(u)\varphi\geq 0.
		\end{equation*}
		Therefore, we conclude that 
		$$2\langle u,\varphi\rangle-2p\mu B^\prime(u)\varphi-2q\lambda^*A^{\prime}(u)\varphi\geq 0\quad\text{for all}\ \varphi\in X_+.$$
		3) Let us prove that
		\begin{equation}
			2\langle u,\varphi\rangle-2p\mu B^\prime(u)\varphi-2q\lambda^*A^{\prime}(u)\varphi=0\quad\text{for all}\ \varphi\in X.
		\end{equation}
		Define the function $\psi=(u+\epsilon\varphi)^+\in X_+$ with $\epsilon>0$ and $\varphi\in X$. Now, we can use the previous which implies that 
		$2\langle u,\psi\rangle-2p\mu B^\prime(u)\psi-2q\lambda^*A^{\prime}(u)\psi \geq 0$. In particular, we see that 
		\begin{eqnarray}
			0\leq\int_{\{u+\epsilon\varphi>0\}}2(\nabla u\nabla(u+\epsilon\varphi)+V(x)u(u+\epsilon\varphi))&-&2q\lambda^*(I_{\alpha_1}*a|u|^q)a(x)|u|^{q-2}u(u+\epsilon\varphi) \nonumber \\ &-&2p\mu(I_{\alpha_2}*|u|^2)|u|^{p-2}u(u+\epsilon\varphi) dx. \nonumber
		\end{eqnarray}
		It follows from the last estimate that
		\begin{eqnarray}
			0&\leq&2\|u\|^2-2q\lambda^*A(u)-2p\mu B(u)\nonumber \\ 
			&+&\epsilon\left[\int_{\mathbb{R}^N}2(\nabla u\nabla\varphi+V(x)u\varphi)-2q\lambda^*(I_{\alpha_1}*a|u|^q)a(x)|u|^{q-2}u\varphi-2p\mu(I_{\alpha_2}*|u|^p)|u|^{p-2}u\varphi dx\right] \nonumber \\
			&-&\int_{\{u+\epsilon\varphi\leq0\}} 2(\nabla u\nabla(u+\epsilon\varphi)+V(x)u(u+\epsilon\varphi))-2q\lambda^*(I_{\alpha_1}*a|u|^q)a(x)|u|^{q-2}u(u+\epsilon\varphi)dx\nonumber \\
			&+&\int_{\{u+\epsilon\varphi\leq0\}}2p(I_{\alpha_2}*|u|^p)|u|^{p-2}u(u+\epsilon\varphi)dx. \nonumber
		\end{eqnarray}
		Under these conditions, by using the last estimate, we infer that 
		\begin{eqnarray}
			0&\leq&\epsilon\left[\int_{\mathbb{R}^N}2(\nabla u\nabla\varphi+V(x)u\varphi)-2q\lambda^*(I_{\alpha_1}*a|u|^q)a(x)|u|^{q-2}u\varphi-2p\mu(I_{\alpha_2}*|u|^p)|u|^{p-2}u\varphi dx\right]\nonumber \\
			&-&2\epsilon\int_{\{u+\epsilon\varphi\leq0\}}\nabla u\nabla\varphi+V(x)u\varphi dx.
		\end{eqnarray}
		Using the same argument as in the proof of Proposition \ref{weak1} we deduce that $\chi_{[u+\epsilon\varphi\leq0]}\rightarrow0$ as $\epsilon\rightarrow0$. Then, using the last estimate and dividing both sides by $\epsilon$, taking the limit as $\epsilon\rightarrow 0$, and applying the Dominated Convergence Theorem, we obtain that
		\begin{equation}
			0\leq\int_{\mathbb{R}^N}2(\nabla u\nabla\varphi+V(x)u\varphi)-2q\lambda^*(I_{\alpha_1}*a|u|^q)a(x)|u|^{q-2}u\varphi-2p\mu(I_{\alpha_2}*|u|^p)|u|^{p-2}u\varphi dx
		\end{equation}
		for any $\varphi\in X$. Now, taking the direction $-\varphi$, we obtain the desired result.
	\end{proof}
	\begin{cor}\label{cor}
		Suppose $(h_1)-(h_4)$ holds. Assume also that $\lambda = \lambda^*$. Then the problem denoted as $(P_{\lambda^*})$ does not admit any solution $u \in \mathcal{N}^0$.
	\end{cor}
	\begin{proof}
		The proof follows arguing by contradiction. Assume that there exists $u \in \mathcal{N}^0$ with $\lambda = \lambda^*$ such that $u$ is a weak solution for the nonlocal elliptic problem \eqref{p1}. In particular, we obtain that
		\begin{equation*}
			\langle u,\varphi\rangle=\mu B^\prime(u)\varphi+\lambda^*A^\prime(u)\varphi \quad\text{for all}\ \varphi\in X.
		\end{equation*}
		Since $\Lambda_n(u) = \lambda^*$ we can apply Lemma \ref{2.10} showing that
		\begin{equation*}
			0=2\langle u,\varphi\rangle-2q\lambda^* A^{\prime}(u)\varphi-2p\mu B^\prime(u)\varphi \quad\text{for all}\ \varphi\in X.
		\end{equation*} 
		As a consequence, we obtain that
		\begin{equation*}
			(2-2q)\lambda^*A^\prime(u)\varphi+(2-2p)\mu B^\prime(u)\varphi=0.
		\end{equation*} 
		The last identity can be viewed as follows
		\begin{equation*}
			\int_{\mathbb{R}^N}[(2-2q)\lambda^*(I_{\alpha_1}*a|u|^q)a(x)|u|^{q-2}u+(2-2p)\mu(I_{\alpha_2}*|u|^p)|u|^{p-2}u]\varphi dx=0.
		\end{equation*}
		Recall also that $u>0$ a.e. in $\mathbb{R}^N$. In particular, we obtain that
		\begin{equation*}
			(2-2q)\lambda^*(I_{\alpha_1}*a|u|^q)a(x)|u|^{q-2}u=(2p-2)\mu(I_{\alpha_2}*|u|^p)|u|^{p-2}u\ \text{a.e in }\mathbb{R}^N.
		\end{equation*} 
		Therefore, we obtain that
		\begin{equation}\label{a(x)}
			a(x)=\frac{C(I_{\alpha_2}*|u|^p)|u|^{p-q}(x)}{(I_{\alpha_1}*a|u|^q)}.
		\end{equation}
		
		Now, as was proved in \cite[Lemma 5.1]{CPAA} and taking into account that $\alpha_2 \in (0, N)$, $2_{ \alpha_2} < p < 2^*_{\alpha_2}$, there exists a constant $\tilde{C} > 0$ in such way that
		\begin{equation}
			\left|\frac{(I_{\alpha_2}*|u|^p(x))}{I_{\alpha_2}(x)}-\|u\|^p_p \right|\leq \tilde{C}\|u\|^p_p.
		\end{equation}
		Under these conditions, we infer that there exists $C > 0$ such that $(I_{\alpha_2}*|u|^p) \leq C I_{\alpha_2}(x)\|u\|^p_p$ holds. In fact, we observe that
		\begin{equation*}
			\left|\frac{(I_{\alpha_2}*|u|^p(x))}{I_{\alpha_2}(x)} \right|=\left|\frac{(I_{\alpha_2}*|u|^p(x))}{I_{\alpha_2}(x)} -\|u\|^p_p+\|u\|^p_p \right|\leq C\|u\|^p_p+\|u\|^p_p.
		\end{equation*}
		On the other hand, by using \cite[Appendix A.4]{Moroz1},  given any function $f\in L^1_{\text{loc}}(\mathbb{R}^N)$ where $f$ is positive on a set of positive measure in $\mathbb{R}^N$, we can conclude that $I_{\alpha}*f$ is strictly positive everywhere in $\mathbb{R}^N$ with $\alpha \in (0, N)$. Furthermore, for each $x\in\mathbb{R}^N$, we infer that
		\begin{equation}\label{con}
			(I_{\alpha}*f)(x)\geq\frac{A_{\alpha}}{(2|x|)^{N-\alpha}}\int_{B_{2|x|}(x)}f(y)dy\geq \frac{C}{|x|^{N-\alpha}}.
		\end{equation}
		It is important to emphasize that the function $a|u|^q$ is positive a.e. in $\mathbb{R}^N$. Furthermore, by using assumption $(h_4)$, we see that $a\in L^{2/(2-q)}(\mathbb{R}^N)$. The last assertion implies that $a|u|^q\in L^1(\mathbb{R}^N)$. Hence,  using \eqref{con} and \eqref{a(x)}, we deduce that
		\begin{equation}
			a(x)=\frac{C(I_{\alpha_2}*|u|^p)|u|^{p-q}(x)}{(I_{\alpha_1}*a|u|^q)}\leq C(I_{\alpha_2}*|u|^p)|x|^{N-\alpha_1}|u|^{p-q}(x).
		\end{equation}
		As a consequence, we infer also that
		\begin{equation*}
			\int_{\mathbb{R}^N}[a(x)]^{\gamma} dx\leq\int_{\mathbb{R}^N} C(I_{\alpha_2}*|u|^p)^{\gamma}|x|^{\gamma(N-\alpha_1)}|u|^{\gamma(p-q)}(x)dx.
		\end{equation*}
		Now, the main objective is to ensure that the integral in the right hand side for the expression just above is finite. It is important to observe that
		\begin{equation}\label{ai}
			\int_{\mathbb{R}^N}[a(x)]^{\gamma} dx\leq\int_{B_R} C(I_{\alpha_2}*|u|^p)^{\gamma}|x|^{\gamma(N-\alpha_1)}|u|^{\gamma(p-q)}(x)dx+ \int_{\mathbb{R}^N\setminus B_R} C(I_{\alpha_2}*|u|^p)^{\gamma}|x|^{\gamma(N-\alpha_1)}|u|^{\gamma(p-q)}(x)dx.
		\end{equation}
		
		Now we shall split the proof into parts. Namely, we consider the following estimates:
		
		1) In the first part we shall consider the integral given just above in the set $B_R$. It is not hard to verify that 
		\begin{equation}
			\int_{B_R} (I_{\alpha_2}*|u|^p)^{\gamma}|x|^{\gamma(N-\alpha_1)}|u|^{\gamma(p-q)}(x)dx\leq C_R\int_{B_R} (I_{\alpha_2}*|u|^p)^{\gamma}|u|^{\gamma(p-q)}(x)dx.
		\end{equation}
		Now, by using the Hölder inequality Hardy-Littlewood-Sobolev inequality, see Lemma \ref{impotante}, and the continuous Sobolev embedding $X\hookrightarrow L^r(\mathbb{R}^N), r\in[2,2^*]$, we obtain that 
		\begin{eqnarray}\label{81}
			\int_{B_R}(I_{\alpha_2}*|u|^p)^{\gamma}|u|^{(p-q)\gamma}&\leq& \left(\int_{B_R}(I_{\alpha_2}*|u|^p)^{\gamma\sigma}dx\right)^{\frac{1}{\sigma}}\left(\int_{B_R}|u|^{\frac{2^*(p-q)\gamma}{(p-q)\gamma}}dx\right)^{\frac{(p-q)\gamma}{2^*}} \nonumber \\
			&=& C_1\|I_{\alpha_2}*|u|^p\|^{\gamma}_{\gamma\sigma}\|u\|^{(p-q)\gamma}_{2*}
			\leq C_2\|I_{\alpha_2}*|u|^p\|^{\gamma}_{\gamma\sigma}\|u\|^{(p-q)\gamma}\leq C_3\||u|^p\|^{\gamma}_s.
		\end{eqnarray}
		Recall also that $\gamma\sigma=Ns/(N-s\alpha_2 )$. Under these conditions, we need to prove that $s \in (1, N/\alpha_2)$ and $sp \in [1, 2^*]$. In order to do that we observe that 
		\begin{equation}\label{sigma1}
			\frac{2^*}{p-q}>\gamma \ \ \text{where} \ \  \sigma=\left(\frac{2^*}{(p-q)\gamma}\right)^\prime>1.
		\end{equation}
		In other words, we mention that 
		\begin{equation*}
			\gamma\sigma=\frac{Ns}{N-\alpha_2 s} \ \ \text{or} \ \  s=\frac{\gamma\sigma N}{N+\gamma\sigma\alpha_2}.\ \ 
		\end{equation*}
		Furthermore, we observe that 
		\begin{equation}\label{84}
			s=\frac{\gamma\sigma N}{N+\gamma\sigma\alpha_2}>1 \ \text{if and only if} \ \gamma>\frac{N}{\sigma(N-\alpha_2)}.
		\end{equation}
		According to \eqref{sigma1} we kwow that $\sigma=2^*/[2^*-(p-q)\gamma]$. In light of \eqref{84} and taking into account the last identity we obtain that
		\begin{equation}\label{ree}
			\gamma>\frac{2^* N}{2^*(N-\alpha_2)+N(p-q)}.
		\end{equation}
		It is not hard to see that the condition $s < N/\alpha_2$ is satisfied. It is important to mention  that $s = s(N, p, q, \gamma, \alpha_2)$. In fact, we can prove that
		\begin{equation*}
			s=\frac{2^*N\gamma}{N[2^*-(p-q)\gamma]+\gamma\alpha_2 2^*}.
		\end{equation*}
		Now, by using the continuous Sobolev embedding $X\hookrightarrow L^r(\mathbb{R}^N)$, $r = ps\in[1,2^*]$, we infer that
		\begin{equation*}
			ps\geq 1 \ \text{if and only if}\ \gamma\geq\frac{2^*N}{2^*Np+N(p-q)-2^*\alpha_2}.
		\end{equation*}
		On the other hand, we obtain the following assertion
		\begin{equation*}
			ps\leq 2^* \ \text{if and only if} \ \gamma\leq\frac{2^*N}{pN+N(p-q)-2^*\alpha_2}.
		\end{equation*}
		It is important to stress that $(2p-q)(N-2)/2 > \alpha_2$ which finishes the estimate in the first part. 
		
		\noindent 2) In this part we shall estimate the integral in the right hand side  for the identity \eqref{ai} in the set $\mathbb{R}^N\setminus B_R$. Firstly, by using the fact that $(I_{\alpha_2}*|u|^p)\leq C I_{\alpha_2}(x)\|u\|^p_p$ and \eqref{con}, we deduce that
		\begin{eqnarray}
			\int_{\mathbb{R}^N\setminus B_R}(I_{\alpha_2}*|u|^p)^{\gamma}|x|^{(N-\alpha_1)\gamma}|u|^{(p-q)\gamma}(x) dx&\leq& C\|u\|^{p\gamma}_p\int_{\mathbb{R}^N\setminus B_R}\frac{1}{|x|^{(N-\alpha_2)\gamma}}|x|^{(N-\alpha_1)\gamma}|u|^{(p-q)\gamma} dx \nonumber \\
			&=& C\|u\|^{p\gamma}_p\int_{\mathbb{R}^N\setminus B_R}\frac{1}{|x|^{(\alpha_1-\alpha_2)\gamma}}|u(x)|^{(p-q)\gamma} dx.
		\end{eqnarray}
		Once again, applying the Hölder inequality, the integral in right side hand in the expression just above can be estimate as follows:
		\begin{equation}
			\int_{\mathbb{R}^N\setminus B_R}\frac{1}{|x|^{(\alpha_1-\alpha_2)\gamma}}|u(x)|^{(p-q)\gamma} dx\leq C \left(\int_{\mathbb{R}^N\setminus B_R}\frac{1}{|x|^{(\alpha_1-\alpha_2)\gamma\sigma}}dx\right)^{\frac{1}{\sigma}}\left(\int_{\mathbb{R}^N\setminus B_R}|u|^{2^*}dx\right)^{\frac{(p-q)\gamma}{2^*}}.
		\end{equation}
		Furthermore, by using the Co-area Theorem \cite{lieb}, we obtain that
		\begin{equation}
			\left(\int_{\mathbb{R}^N\setminus B_R}\frac{1}{|x|^{(\alpha_1-\alpha_2)\gamma\sigma}}dx\right)^{\frac{1}{\sigma}}\left(\int_{\mathbb{R}^N\setminus B_R}|u|^{2^*}dx\right)^{\frac{(p-q)\gamma}{2^*}}\leq C\|u\|^{(p-q)\gamma}_{2^*}\left[\int^{+\infty}_R \frac{1}{z^{(\alpha_1-\alpha_2)\gamma\sigma}}z^{N-1}dz\right]^{\frac{1}{\sigma}}.
		\end{equation}
		It is important to mention that the last integral is finite if and only if $N<\gamma\sigma(\alpha_1-\alpha_2)$. Under these conditions, we obtain that
		\begin{equation*}
			\frac{2^* N}{2^*(\alpha_1-\alpha_2)+N(p-q)}<\gamma \quad \text{where} \quad \alpha_1>\alpha_2 \quad \text{and} \quad p<\frac{2\alpha_1}{N-2}.
		\end{equation*}
		Furthermore, we observe that $\|u\|^{p\gamma}_p$ is finite if and only if $2\leq p\leq 2^*$. Recall also that $2\alpha_1/(N-2)<2^*$ where $\alpha_1>N-2$. Hence, the last estimate implies that $2\leq p<2\alpha_1/(N-2)$. Notice also that $2_{\alpha_2}<p<2^*_{\alpha_2}$ which implies that $2<2^*_{\alpha_2}$ is verified if and only if $N-4<\alpha_2$. In other words, we choose $\alpha_2\in(N-4,N)$. As a consequence, we assume that  $p\in(2,\min\left\{2^*_{\alpha_2},2\alpha_1/(N-2)\right\})$. Recall also that $(2p-q)(N-2)/2 > \alpha_2$. Hence, $\alpha_2>N-4$ which implies that $2(N-4)/(N-2)<2p-q$. Therefore, we need to choose $\alpha_1\in(N-2,N)$, $p\in[2,\min\left\{2^*_{\alpha_2},2\alpha_1/(N-2)\right\})$, $\alpha_2\in(N-4,\min\left\{N,(2p-q)(N-2)/2\right\})$, $2(N-4)/(N-2)<2p-q$. Furthermore, we also assume that
		\begin{equation}
			\frac{2^* N}{2^*(\alpha_1-\alpha_2)+N(p-q)} <\gamma<\min\left\{\frac{2^*}{p-q},\frac{2^*N}{pN+N(p-q)-2^*\alpha_2}\right\}.
		\end{equation}
		Under these conditions, we obtain that
		\begin{equation}
			\int_{\mathbb{R}^N}[a(x)]^{\gamma} dx<+\infty.
		\end{equation}
		This leads to a contradiction with hypothesis $(h_4)$ with $\gamma=r$ where $a$ does not belong to $L^r(\mathbb{R}^N)$. Hence, the problem $(P_{\lambda^*})$ does not admit any solution $u\in\mathcal{N}^0$. This finishes the proof.
	\end{proof}

	\begin{lem}\label{solu}
		Suppose $(h_1)-(h_4)$ holds. Assume also that $\lambda = \lambda*$. Then the problem $(P_{\lambda^*})$ admits at least two solutions $w_{\lambda^*}\in\mathcal{N}^-$ and $u_{\lambda^*}\in\mathcal{N}^+$.
	\end{lem}
	\begin{proof}
		Firstly, we shall prove the existence of a solution  $w_{\lambda^*}\in\mathcal{N}^-$. Consider $(\lambda_k)_{k\in\mathbb{N}}\in (0,\lambda^*)$ such that $\lambda_k\rightarrow\lambda^*$. Let $(w_{\lambda_k})_{k\in\mathbb{N}}\subset\mathcal{N}^-_{\lambda_k}$ be a weak solution to problem \eqref{p1} for each $\lambda_k \in (0, \lambda^*)$, see Proposition \ref{weak1}. Therefore, we obtain that $J_{\lambda_k}(w_{\lambda_k})=c_{\mathcal{N}^-_{\lambda_k}} $, $J^\prime(w_{\lambda_k})\varphi=0, \varphi \in X$ and $J^{\prime\prime}(w_{\lambda_k})(w_{\lambda_k},w_{\lambda_k})<0$. Recall that $c_{\mathcal{N}^-_{\lambda_k}}$ is given by \eqref{c1} with $\lambda = \lambda_k$. Here the energy functional with $\lambda=\lambda_k$ is denoted by $J_{\lambda_k}$. Now, we assume by contradiction that $\|w_{\lambda_k}\|\rightarrow+\infty$. Notice that
		\begin{equation}
			J_{\lambda_k}(w_{\lambda_k})=c_{\mathcal{N}^-_{\lambda_k}}=\left(\frac{1}{2}-\frac{1}{2p}\right)\|w_{\lambda_k}\|^2+\left(\frac{1}{2p}-\frac{1}{2q}\right)\lambda_k A(w_{\lambda_k}).
		\end{equation}
		Using the Hardy-Littlewood-Sobolev inequality, see Proposition \ref{hls}, and that the continuous Sobolev embedding $X\hookrightarrow L^r(\mathbb{R}^N), r\in[2,2^*]$, we infer that
		\begin{equation}
			c_{\mathcal{N}^-_{\lambda_k}}\geq\left(\frac{1}{2}-\frac{1}{2p}\right)\|w_{\lambda_k}\|^2+\left(\frac{1}{2p}-\frac{1}{2q}\right)C_3\lambda_k\|a\|^2_s\|w_{\lambda_k}\|^{2q}.
		\end{equation}
		On the other hand, we observe that $\lambda \mapsto c_{\mathcal{N}^-}$ is a decreasing function, see \cite[Proposition 3.1 and Proposition 3.3]{fractional}. In particular, using the last estimate and taking into account that $\|w_{\lambda_k}\|\rightarrow+\infty$, we obtain that there exists $C > 0$ such that  $$+\infty > C \geq \displaystyle\lim_{k\rightarrow+\infty}c_{\mathcal{N}^-_{\lambda_k}}=\displaystyle\lim_{k\rightarrow+\infty}J(w_{\lambda_k})\geq+\infty.$$
		This is a contradiction proving that the sequence $(w_{\lambda_k})_{k\in\mathbb{N}}$ is a bounded in $X$. Hence,  $w_{\lambda_k}\rightharpoonup w_{\lambda^*}$ in $X$ for some  $w_{\lambda^*} \in X$. As a consequence, we obtain that $w_{\lambda_k}\rightarrow w_{\lambda^*}$ in $L^r(\mathbb{R}^N)$ for any $r\in[2,2^*)$. Recall also that $w_{\lambda_k}\rightarrow w_{\lambda^*}$ almost everywhere in $\mathbb{R}^N$ and there exist $h_r \in L^r(\mathbb{R}^N)$ such that $|w_{\lambda_k}| \leq h_r$.
		Notice also that $(w_{\lambda_k})_{k\in\mathbb{N}}$ is a weak solution to Problem \eqref{p1} for each $\lambda_k$. Hence, we deduce the following assertion:
		\begin{equation}\label{fracal}
			\langle w_{\lambda_k},\varphi\rangle-\mu B^\prime(w_{\lambda_k})\varphi=\lambda_k A^\prime(w_{\lambda_k})\varphi, \,\, \varphi \in X.
		\end{equation}
		Now, using the same ideas discussed in the proof of Lemma \ref{2.10} and taking into account the Dominated Convergence Theorem and Fatou's Lemma, we infer that 
		\begin{equation*}
			\langle w_{\lambda^*},\varphi\rangle-\mu B(w_{\lambda^*})\varphi\geq\lambda^* A^\prime(w_{\lambda^*})\varphi, \,\, \mbox{for all} \,\, \varphi \in X.
		\end{equation*}
		Furthermore, by using the testing functions $w_{\lambda_k}=\varphi$ and $w_{\lambda^*}=\varphi$, we deduce that 
		\begin{eqnarray*}
			\langle w_{\lambda_k},w_{\lambda_k}\rangle-\mu B^\prime(w_{\lambda_k})w_{\lambda_k}=\lambda_k A^\prime(w_{\lambda_k})w_{\lambda_k}, \quad
			\langle w_{\lambda_k},w_{\lambda^*}\rangle-\mu B^\prime(w_{\lambda_k})w_{\lambda^*}=\lambda_k A^\prime(w_{\lambda_k})w_{\lambda^*}.
		\end{eqnarray*}
		As a consequence, we see that 
		\begin{equation*}
			\langle w_{\lambda_k},w_{\lambda_k}-w_{\lambda^*}\rangle-\mu B^\prime(w_{\lambda_k})w_{\lambda_k}+\mu B^\prime(w_{\lambda_k})w_{\lambda^*}=\lambda_k A^\prime(w_{\lambda_k})w_{\lambda_k}-\lambda_k A^\prime(w_{\lambda_k})w_{\lambda^*}.
		\end{equation*}
		Using the last estimate we obtain that 
		\begin{eqnarray}
			\limsup_{k\rightarrow+\infty}\langle w_{\lambda_k}, w_{\lambda_k}- w_{\lambda^*}\rangle&\leq&\limsup_{k\rightarrow+\infty}(\lambda_k A^\prime(w_{\lambda_k})w_{\lambda_k})+\limsup_{k\rightarrow+\infty}(-\lambda_k A^\prime(w_{\lambda_k})w_{\lambda^*}) \nonumber \\
			&=&\lambda^* A^{\prime}(w_{\lambda^*})w_{\lambda^*}-\liminf_{k\rightarrow +\infty}\lambda_k A^\prime(w_{\lambda_k})w_{\lambda^*} 
			\leq\lambda^* A^{\prime}(w_{\lambda^*})w_{\lambda^*}-\lambda^* A^{\prime}(w_{\lambda^*})w_{\lambda^*}=0.
		\end{eqnarray}
		Hence, we deduce that 
		\begin{equation}
			\limsup_{k\rightarrow+\infty}\|w_{\lambda_k}-w_{\lambda^*}\|^2\leq\limsup_{k\rightarrow+\infty}\langle w_{\lambda_k},w_{\lambda_k}-w_{\lambda^*} \rangle+\limsup_{k\rightarrow+\infty}-\langle w_{\lambda^*}, w_{\lambda_k}-w_{\lambda^*}\rangle\leq 0.
		\end{equation}
		In particular, we obtain that $w_{\lambda_k}\rightarrow w_{\lambda^*}$ in $X$. Therefore, we obtain that 
		\begin{eqnarray*}
			J^\prime(w_{\lambda^*})w_{\lambda^*}=0, \quad
			J^{\prime\prime}(w_{\lambda^*})(w_{\lambda^*},w_{\lambda^*})\leq 0.
		\end{eqnarray*}
		Notice also that $w_{\lambda^*} \neq 0$. Recall also that the function $w_{\lambda^*}$ satisfies
		\begin{eqnarray*}
			\|w_{\lambda^*}\|^2-\lambda^* A(w_{\lambda^*})-\mu B(w_{\lambda^*})=0, \quad
			\langle w_{\lambda^*},\varphi\rangle -\mu B^\prime(w_{\lambda^*})\varphi-\lambda^* A^\prime(w_{\lambda^*})\varphi\geq 0, \,\, \mbox{for all} \,\, \varphi\geq 0.
		\end{eqnarray*}
		As a consequence, we obtain that $w_{\lambda^*}\in \mathcal{N}^-\cup\mathcal{N}^0$. Now, arguing as was done in the proof of Proposition \ref{weak1}, we prove that $w_{\lambda^*}$ is a weak solution to the problem \eqref{p1} where $\lambda=\lambda^*$. Furthermore, by using Corollary \ref{cor}, we also infer that $w_{\lambda^*}\in\mathcal{N}^-$. Moreover, we observe that
		\begin{equation}
			J_{\lambda^*}(w_{\lambda^*})=\lim_{k\rightarrow+\infty}J_{\lambda_k}(w_{\lambda_k})=\lim_{k\rightarrow+\infty}c_{\mathcal{N}^-_{\lambda_k}}.
		\end{equation}
		Now, by similar arguments discussed just above we can also obtain the solution $u_{\lambda^*}\in\mathcal{N}^+$ with $\lambda = \lambda^*$. This concludes the proof.
	\end{proof}

	\section{The proof of main results}\label{s3}
	
	\subsection{The proof of Theorem \ref{theorem1}} The proof of Theorem \ref{theorem1} follows by considering the existence of a minimizer for the minimization problem given by \eqref{c2}. According to  Proposition \ref{2.55} we obtain that $\mathcal{N}^0$ is empty for each $\lambda\in(0,\lambda^*)$. For this case, we obtain that $\overline{\mathcal{N}}^+ = \mathcal{N}^+ \cup \{0\}$. Hence, any minimizer for the minimization problem \eqref{c2} belongs to $\mathcal{N}^+$. Here was used that $c_{\mathcal{N}^+} < 0$, see Lemma \ref{2.8}. Recall also that any minimizer $u$ for the minimization problem \eqref{c2} give us a weak solution to the Problem \eqref{p1}, see Proposition \ref{p2.8} and Lemma \ref{p2.8}. This ends the proof.
	
	\subsection{The proof of Theorem \ref{theorem2} (i)} The proof of Theorem \ref{theorem2} follows proving that the minimization problem \eqref{c1} is attained by some function $u \in \mathcal{N}^-$. In order to do that, by using Lemma \ref{2.5} and Proposition \ref{weak1}, for each $\lambda\in(0,\lambda_*)$ there exists $v\in\mathcal{N}^-$ such that $v$ is a weak solution to the nonlocal elliptic problem \eqref{p1}. Here we assume that $\lambda \in (0, \lambda_*)$. Under these conditions, we know that $t^{n,-}(v) > t_e(v)$, see Figures \ref{figure1} and \ref{figure2} . As a consequence, $\lambda = R_n(t^{n,-}(v)v) < R_e(t^{n,-}(v) v)$, see Remark \ref{2.4} and Figure \ref{figure2}. Hence, by using Remark \ref{2o1}, we obtain that $c_{\mathcal{N}^-} = J(v) = J(t^{n,-}(v)v) > 0$. This ends the proof for the item $(i)$.
	
	\subsection{The proof of Theorem \ref{theorem2}(ii)} Here we assume that $\lambda=\lambda_*$. Notice that the problem \eqref{p1} has at least a weak solution $v\in\mathcal{N}^-$, see Lemma \ref{2.5} and Proposition \ref{weak1}. Furthermore, by using Remark \ref{lambda_e} together with Lemma \ref{lema2.1}, there exists $w\in X\setminus\{0\}$ such that $\lambda_*=\Lambda_e(w)=R_e(t_e(w)w)=R_n(t^{n,-}(w)w)$. Therefore, by using Remark \ref{2.4}, we know that $R_e(tw)=R_n(tw)$ if and only if $t=t_e(w)$. In particular, we observe that $t_e(w)=t^{n,-}(w)$. Thus, $R_n(t^{n,-}(w)w)=R_e(t^{n,-}(w)w)=\lambda_*$. Now, by using Remark \ref{2o1}, we obtain that $J(t^{n,-}(w)w)=0$.  Furthermore, by using the fact that $t^{n,-}(w)w\in\mathcal{N}^-$, we have $c_{\mathcal{N}^-}=J(v)\leq J(t^{n,-}(w)w)=0$. On the other hand, we know that $\lambda=\lambda_*=\inf_{u\in X\setminus\{0\}}\Lambda_e(u)\leq\Lambda_e(v)=R_e(t_e(v)v)$. Recall also that $\lambda_*=R_n(t^{n,-}(v)v)=R_n(v)\leq R_e(t_e(v)v)$. The last assertion implies that $t_e(v)\leq 1=t^{n,-}(v)$. In particular, we have $\lambda_* = R_n(v)\leq R_e(v)$. Therefore, we obtain that $J(v)\geq 0$. Hence, using the estimate given above, we infer that $c_{\mathcal{N}} = J(v)=0$.

	\subsection{The proof of Theorem \ref{theorem2}(iii)}
	
	Here we assume that $\lambda \in (\lambda_*,\lambda^*)$. Under these conditions, we know that the problem \eqref{p1} has at least one weak solution $v\in\mathcal{N}^-$, see Lemma \ref{2.5} and Proposition \ref{weak1}. Let us fix a function $u\in X\setminus \{0\}$ such that $\lambda_*\leq\Lambda_e(u)=R_e(t_e(u)u)<\lambda$. Notice also that $\lambda=R_n(t^{n,-}(u)u)$. Hence, we obtain that $t^{n,-}(u)\in(0,t_e(u))$ and $R_e(tu)<R_n(tu)$, see Figure \ref{figure2}. In particular, for $t=t^{n,-}(u)$, we obtain that $R_e(t^{n,-}(u)u)<R_n(t^{n,-}(u)u)=\lambda$. Therefore, by using Remark \ref{2o1} and taking into account that $t^{n,-}(u)u\in\mathcal{N}^-$, we obtain that $c_{\mathcal{N}^-}\leq J(t^{n,-}(u)(u)) < 0$.
	
	\subsection{The proof of Theorem \ref{theorem3}} Firstly, we observe that does not exist any weak solution $u \in \mathcal{N}^0$ for $\lambda = \lambda^*$ for the problem \eqref{p1}, see Corollary \ref{cor}. Hence, the minimization problems given by \eqref{c1} and \eqref{c2} are attained by functions $u \in \mathcal{N}^+$ and $v \in \mathcal{N}^-$ with $\lambda = \lambda^*$. Under these conditions, we obtain $u$ and $v$ are weak solutions to the nonlocal elliptic problem \eqref{p1}, see Lemma \ref{solu}. This ends the proof.



\end{document}